\documentclass[10pt, a4paper]{article}
\usepackage{graphicx,latexsym, amsmath,amsfonts}
\usepackage{amsthm}
\usepackage{enumerate}
\usepackage{bbm}
\usepackage{subfig}

\usepackage{epsfig}
\usepackage{color}

\begin{document}

	
	\renewcommand{\d}{d}
	
	\newcommand{\E}{\mathbb{E}}
	\newcommand{\PP}{\mathbb{P}}
	\newcommand{\DD}{\mathbb{D}}
	\newcommand{\R}{\mathbb{R}}
	\newcommand{\cD}{\mathcal{D}}
	\newcommand{\cF}{\mathcal{F}}
	\newcommand{\cK}{\mathcal{K}}
	\newcommand{\N}{\mathbb{N}}
	\newcommand{\fracs}[2]{{ \textstyle \frac{#1}{#2} }}
	\newcommand{\sign}{\operatorname{sign}}
	
	\newtheorem{theorem}{Theorem}[section]
	\newtheorem{lemma}[theorem]{Lemma}
	\newtheorem{coro}[theorem]{Corollary}
	\newtheorem{defn}[theorem]{Definition}
	\newtheorem{assp}[theorem]{Assumption}
	\newtheorem{expl}[theorem]{Example}
	\newtheorem{prop}[theorem]{Proposition}
	\newtheorem{proposition}[theorem]{Proposition}
	\newtheorem{corollary}[theorem]{Corollary}
	\newtheorem{rmk}[theorem]{Remark}
	\newtheorem{notation}[theorem]{Notation}

	\def\a{\alpha} \def\g{\gamma}
	\def\e{\varepsilon} \def\z{\zeta} \def\y{\eta} \def\o{\theta}
	\def\vo{\vartheta} \def\k{\kappa} \def\l{\lambda} \def\m{\mu} \def\n{\nu}
	\def\x{\xi}  \def\r{\rho} \def\s{\sigma}
	\def\p{\phi} \def\f{\varphi}   \def\w{\omega}
	\def\q{\surd} \def\i{\bot} \def\h{\forall} \def\j{\emptyset}
	
	\def\be{\beta} \def\de{\delta} \def\up{\upsilon} \def\eq{\equiv}
	\def\ve{\vee} \def\we{\wedge}
	
	\def\t{\tau}
	
	\def\F{{\cal F}}
	\def\T{\tau} \def\G{\Gamma}  \def\D{\Delta} \def\O{\Theta} \def\L{\Lambda}
	\def\X{\Xi} \def\S{\Sigma} \def\W{\Omega}
	\def\M{\partial} \def\N{\nabla} \def\Ex{\exists} \def\K{\times}
	\def\V{\bigvee} \def\U{\bigwedge}
	
	\def\1{\oslash} \def\2{\oplus} \def\3{\otimes} \def\4{\ominus}
	\def\5{\circ} \def\6{\odot} \def\7{\backslash} \def\8{\infty}
	\def\9{\bigcap} \def\0{\bigcup} \def\+{\pm} \def\-{\mp}
	\def\<{\langle} \def\>{\rangle}
	
	\def\lev{\left\vert} \def\rev{\right\vert}
	\def\1{\mathbf{1}}

	\newcommand\wD{\widehat{\D}}
	\newcommand\EE{\mathbb{E}}
	
	\newcommand{\ls}[1]{\textcolor{red}{\tt Lukas: #1 }}

	\title{ \bf {Sharp $L^1$-Approximation of the log-Heston SDE by Euler-type methods}}

	\author{Annalena Mickel \footnote{Mathematical Institute and   DFG Research Training Group 1953, 
			University of Mannheim,  B6, 26, D-68131 Mannheim, Germany,  \texttt{amickel@mail.uni-mannheim.de}}
		\and   Andreas Neuenkirch  \footnote{Mathematical Institute, 
			University of Mannheim,  B6, 26, D-68131 Mannheim, Germany, \texttt{aneuenki@mail.uni-mannheim.de}} 
	}

	\date{\today}

	\maketitle
	
	\begin{abstract} 
		We study the $L^1$-approximation of the log-Heston SDE at equidistant time points by  Euler-type methods.
		We establish the convergence order $ 1/2-\epsilon$ for $\epsilon >0$ arbitrarily small,
		if the Feller index $\nu$ of the underlying CIR process satisfies  $\nu > 1$. Thus, we recover the standard convergence order of the Euler scheme for SDEs with globally Lipschitz coefficients.
		Moreover, we discuss the case $\nu \leq 1$ and illustrate our findings by several numerical
		examples.

		\medskip
		\noindent \textsf{{\bf Key words: } \em CIR process, Heston model, Euler-type methods, $L^1$-error \\
		}
		
	\end{abstract}
	
	\section{Introduction and Main Results}
	The CIR process originates from the works of Feller in the 1950s, see e.g. \cite{Feller},  and was used by Cox, Ingersoll and Ross \cite{CIR} to model short term interest rates. It is the solution to the following stochastic differential equation (SDE)
	\begin{equation}\label{CIR}
		dV_t = \kappa(\theta - V_t) dt + \sigma \sqrt{V_t} dW_t,  \qquad t\in[0,T],
	\end{equation} 
	where $V_0=v_0>0$ and $W=(W_t)_{t \in [0,T]}$ is a Brownian motion. The parameters can be interpreted as follows: $\theta>0$ is the long run mean of the process, $\kappa>0$ is its speed of mean reversion and $\sigma>0$ is its volatility. We assume the initial value $v_0$ to be deterministic. We denote the Feller index of the CIR process by
	\begin{displaymath}
		\nu := \frac{2\kappa\theta}{\sigma^2}.
	\end{displaymath}
	Note that the CIR process takes positive values only. Additionally, if $\nu \geq 1$ almost all sample paths are strictly positive. 
	The CIR process is used in particular to model the volatility of the asset price in the Heston model \cite{heston}. Here, the SDE for the price process and its volatility are given by 
	\begin{equation}  
		\begin{aligned}
			dS_t &= \mu S_t dt + \sqrt{V_t} S_t  \left(\rho dW_t + \sqrt{1-\rho^2} dB_t\right), \\
			dV_t &= \kappa(\theta - V_t) dt + \sigma \sqrt{V_t} dW_t, \label{hes-eq}
		\end{aligned}\qquad t\in[0,T],
	\end{equation}
	where $S_0=s_0>0$ is deterministic, $\mu \in \mathbb{R}$ is the risk-free interest rate,  $\rho \in [-1,1]$ determines the correlation between the two processes and $W=(W_t)_{t \in [0,T]}$,  $B=(B_t)_{t \in [0,T]}$ are independent Brownian motions.  Usually, the log-Heston model  instead of the Heston model is considered in numerical practice. This yields the SDE
	\begin{equation} 
		\begin{aligned}
			dX_t &= \left(\mu-\frac12 V_t\right) dt + \sqrt{V_t} \left(\rho d W_t + \sqrt{1-\rho^2} d B_t\right), \\
			dV_t &= \kappa (\theta-V_t) dt + \sigma \sqrt{V_t} dW_t, \label{hes-log} 
		\end{aligned}\qquad t\in[0,T],
	\end{equation}
	where $X_t=\log(S_t)$ and $x_0=\log(s_0) \in \mathbb{R}$. Since the  right hand side of the SDE for the log-asset price  does not depend on $X_t$, its approximation reduces to the approximation of  a Riemann integral of $V$ and an It\=o integral of $\sqrt{V}$. So, the main difficulty is the approximation of the CIR process $V=(V_t)_{t\in[0,T]}$. Since the CIR process takes positive values only and the diffusion coefficient is a square root  and thus  not globally Lipschitz  continuous, much effort has been devoted to this problem in the last 25 years, see Subsection \ref{survey-CIR}.
	
	In this article, we are looking at Euler discretization schemes for the CIR process and the log-Heston model. We will work with an equidistant discretization
	\begin{displaymath}
		t_k = k \Delta t, \quad k=0, \ldots,N,
	\end{displaymath}
	with $\Delta t =T/N$. A naive Euler discretization will give negative values and is not well defined
	due to the square root coefficient. Therefore a "fix" is required. A summary of the existing Euler schemes for the CIR process and a numerical comparison can be found in \cite{lord-comparison}, where a general framework for Euler schemes for the CIR process is proposed as
	\begin{equation} \label{cir-euler}
		\begin{aligned}
			\bar{v}_{t_{k+1}} &= f_1(\bar{v}_{t_{k}}) + \kappa\left(\theta-f_2(\bar{v}_{t_{k}})\right)(t_{k+1}-t_k)+\sigma\sqrt{f_3(\bar{v}_{t_k})}\left(W_{t_{k+1}}-W_{t_{k}}\right)\\
			\hat{v}_{t_{k+1}} &= f_3(\bar{v}_{t_{k+1}})
		\end{aligned}
	\end{equation}
	for $k=0, \ldots, N-1$
	with $\hat{v}_0=\bar{v}_0=v_0$ and suitable functions $f_i$ that are chosen from
	\begin{displaymath}
		\begin{aligned}
			{\tt id}: \mathbb{R} \rightarrow \mathbb{R},&  \quad {\tt id}(x)=x,	
			\\ 
			{\tt abs}: \mathbb{R} \rightarrow [0, \infty),&  \quad {\tt abs}(x)=x^+, \\
			{\tt  sym}: \mathbb{R} \rightarrow [0, \infty),&  \quad {\tt sym}(x)=|x|.
		\end{aligned}
	\end{displaymath}
	Here we will study the Euler schemes with $f_i$ given by
	\begin{equation} \label{cir-choices1} 
		\begin{aligned}
			f_1={\tt id}, \quad  f_2 \in \{ {\tt id,  abs, sym} \}, \quad 
			f_3 \in \{ {\tt abs, sym} \}
		\end{aligned}
	\end{equation}
	or
	\begin{equation} \label{cir-choices2} 
		\begin{aligned}
			f_1=f_2=f_3 \in \{ \tt abs, sym\}.
		\end{aligned}
	\end{equation}
	The first set of conditions modifies the coefficients of the CIR process to deal with negative values, which may arise in the computation. For example, $\sqrt{v}$ is replaced by $\sqrt{v^+}$ or  $\sqrt{|v|}$. After the approximation $\bar{v}$ has been computed, $f_3$ is again applied to obtain $\hat{v}$, since $\bar{v}$ may be still negative.
	The second set of conditions is different. Here after each Euler step ${\tt abs}$ or ${\tt sym}$, respectively, is applied to avoid negative values. See also Subsection \ref{prelim-schemes1} and Subsection \ref{prelim-schemes2}.
	
	Table \ref{tab:Euler} shows all Euler schemes that are presented in \cite{lord-comparison} in detail. The Full Truncation Euler was introduced in the same paper. The origin of the Euler with Absorption fix is unknown, the Symmetrized Euler was analyzed in \cite{BO} for example. The scheme from Higham and Mao was first analyzed in \cite{HM} and the Partial Truncation Euler was first introduced in \cite{DD}.
	\begin{table}[tbhp]
		{\footnotesize
			\caption{Euler schemes from \cite{lord-comparison}}\label{tab:Euler}
			\begin{center}
				\begin{tabular}{|c|c|c|c|} \hline
					Scheme & $f_1(x)$ & $f_2(x)$ & $f_3(x)$ \\ \hline
					Absorption (AE) & $(x)^+$ & $(x)^+$ & $(x)^+$ \\
					Symmetrized (SE) & $|x|$ & $|x|$ & $|x|$\\ 
					Higham and Mao (HM) & $x$ & $x$ & $|x|$\\ 
					Partial Truncation (PTE) & $x$ & $x$ & $(x)^+$\\ 
					Full Truncation (FTE) & $x$ & $(x)^+$ & $(x)^+$\\ 
					\hline
				\end{tabular}
			\end{center}
		}
	\end{table}
	Results involving a (polynomial) convergence rate for these Euler schemes are rare and usually come along with a strong restriction on the Feller index, see Subsections \ref{survey-CIR} and \ref{survey-logHeston}. In this article, we will prove the $L^1$-convergence rate of $\frac{1}{2}-\epsilon$ for all these schemes if $\nu > 1$ (with $\epsilon >0$ arbitrarily small). Furthermore, we will show that this result carries over to the log-Heston model if the price process is discretized with the standard Euler scheme, i.e. with
	\begin{equation}
		\begin{aligned} \label{heston-euler}
			\hat{x}_{t_{k+1}}=	\hat{x}_{t_{k}} & + \left( \mu - \frac{1}{2} 	\hat{v}_{t_{k}} \right)(t_{k+1}-t_k) \\ & +
			\sqrt{\hat{v}_{t_{k}}} \left(\rho \left(W_{t_{k+1}}-W_{t_{k}}\right)  + \sqrt{1-\rho^2} \left(B_{t_{k+1}}-B_{t_{k}}\right)\right),
		\end{aligned}
	\end{equation}
	where $\hat{x}_{0}=x_0$ and $k=0, \ldots, N-1$.
	
	\begin{theorem}\label{thm:up-bound} Let $\nu>1$, $\epsilon >0$ and $(\hat{v}_{t_k},\hat{x}_{t_{k}})_{k \in \{ 0, \ldots, N \} }$ given by Equations \eqref{cir-euler},  \eqref{cir-choices1}, \eqref{heston-euler} or by Equations \eqref{cir-euler},  \eqref{cir-choices2}, \eqref{heston-euler}. Then we have
		\begin{displaymath}
			\begin{aligned}
				&  \lim_{N \rightarrow \infty} \,\,   N^{1/2-\epsilon}  \left( \max_{k \in \{ 0, \ldots, N \} }   \mathbb{E} \left[   | X_{t_k}- \hat{x}_{t_{k}}|\right] +    \max_{k \in \{ 0, \ldots, N \} }   \mathbb{E}\left[    | V_{t_k}- \hat{v}_{t_{k}}|\right]  \right) =0. 
			\end{aligned}
		\end{displaymath}
	\end{theorem}
	Thus, we recover (up to an arbitrarily small $\epsilon>0$) the standard convergence order of the Euler scheme for SDEs with globally Lipschitz continuous coefficients.

	For the case $\nu \leq 1$ we can obtain e.g.~for the Euler schemes  given by Equations \eqref{cir-euler}, \eqref{cir-choices1}, \eqref{heston-euler} convergence order $\nu/2-\epsilon$.
	However, this estimate does not seem to be sharp, see our simulation study in Section \ref{sec:simulation}.
	\begin{proposition}
		\label{prop:up-bound} Let $\nu \leq 1$, $\epsilon >0$ and $(\hat{v}_{t_k},\hat{x}_{t_{k}})_{k \in \{ 0, \ldots, N \} }$ given by Equations \eqref{cir-euler},  \eqref{cir-choices1}, \eqref{heston-euler}. Then we have
		\begin{displaymath}
			\begin{aligned}
				&  \lim_{N \rightarrow \infty} \,\,   N^{\nu/2-\epsilon}  \left( \max_{k \in \{ 0, \ldots, N \} }   \mathbb{E} \left[   | X_{t_k}- \hat{x}_{t_{k}}|\right] +    \max_{k \in \{ 0, \ldots, N \} }   \mathbb{E}\left[    | V_{t_k}- \hat{v}_{t_{k}}|\right]  \right) =0. 
			\end{aligned}
		\end{displaymath}	
	\end{proposition}
	
	We conclude this section with a summary of previous results from the literature, further new results and an outline of the remainder of this article.

	\subsection{Previous results for the CIR process} \label{survey-CIR}

	The strong approximation of the CIR process  has been intensively studied in the last years.
	The first works on this topic are \cite{DD,AA1,HM}, which prove strong convergence (without a polynomial rate) of various explicit and implicit schemes using the Yamada-Watanabe approach. 
	
	One of the schemes of \cite{AA1} is the drift-implicit square root Euler scheme which is well defined and positivity preserving for $\nu \geq \frac{1}{2}$. This scheme turned out to be accessible to a more detailed error analysis, see \cite{DNS,AA2,NeSz,HuJeNo}.   In particular, for the $L^1$-approximation at the final time point \cite{AA2} establishes convergence order 1 for $\nu >2$, while  \cite{DNS} gives  convergence order $\frac{1}{2}$ for $\nu >1$, \cite{HuJeNo} yields order $\nu-\frac{1}{2}-\epsilon$ for $\nu >\frac{1}{2}$ and \cite{HeHe} establishes order $\frac{1}{2}$ for the particular case $\nu=\frac{1}{2}$. 
	
	A  breakthrough for the (very challenging) case $\nu \leq 1$ was  provided by \cite{HeHe} and \cite{Mario}. In particular, the truncated Milstein scheme of \cite{Mario} attains $L^1$-convergence order $\min \{\frac{1}{2},\nu \}-\epsilon$ in this regime.

	So, which rates are best possible for the (non-adaptive) $L^1$-approximation of the CIR process  at the final time point?
	This question has been answered by the works \cite{HeHeMG} and \cite{JentzenHefter}, which
	yield
	\begin{displaymath}
		\liminf_{N \rightarrow \infty} \, N^{\min \{\nu,1 \}} \, \inf_{u \in \mathcal{U}}	\, \mathbb{E}\left[  \big| u( W_{t_1},W_{t_2}, \ldots, W_{t_N}) -V_T \big|\right]  >0,
	\end{displaymath}
	where $\mathcal{U}$ is the set of measurable functions $u: \mathbb{R}^N \rightarrow \mathbb{R}$.
	Thus,  the convergence rate of the truncated Milstein scheme for $\nu \leq \frac{1}{2}$ and the rate of the drift-implicit square root Euler for $\nu >2$ are optimal.

	In contrast to this, convergence rate results for the explicit Euler schemes in Table \ref{tab:Euler} have been rare. In \cite{BER}, the authors prove  $L^p$-convergence order $\frac{1}{2}$ for the Symmetrized Euler but with a strong restriction on the Feller index. For  FTE the $L^p$-convergence order $\frac{1}{2}$ for $2\le p <\nu-1$ and $\nu>3$ is shown in \cite{CS2}. As mentioned, \cite{lord-comparison} provides a survey and numerical comparison of Euler-type schemes. Further contributions on the strong approximation of the CIR process can be found in  \cite{GR, CHA, MS16, BoOl}, where \cite{MS16} provides in particular a method for the uniform approximation of the CIR process.
	
Clearly, the truncated Milstein scheme and the drift-implicit square root Euler scheme have similar or even
	better convergence properties than the Euler schemes studied in this article. However, Euler schemes for the approximation of the CIR process are nevertheless popular and our analysis addresses a problem that has been open for a long time, namely to establish (sharp) convergence rates for these schemes under mild assumptions on the Feller index. Moreover, for the log-Heston SDE the obtained convergence rate is optimal for $\nu>1$, see the next subsection.

	\subsection{Previous results for the log-Heston SDE} \label{survey-logHeston}

	We are not aware of any results concerning the strong approximation of the log-Heston model except  \cite{survey,Altmeyer}. In \cite{Altmeyer} the drift-implicit square root Euler for the CIR process is combined with an Euler discretization of the log-Heston process and $L^2$-convergence order $\frac{1}{2}$ is obtained for $\nu>2$, while \cite{survey} uses a drift-implicit Milstein discretization of the CIR process instead and obtains $L^2$-convergence for 
	$\nu>1$. 
	
In view of these results, Theorem \ref{thm:up-bound} and Proposition \ref{prop:up-bound} are a major improvement, in particular with respect to the conditions on the Feller index. Moreover, for $\nu>1$ the obtained convergence rate is best possible. In \cite{opt-tbd} we study the optimal $L^1$-approximation of the log-Heston model by equidistant discretizations of $W$ and $B$ and show in particular that
		\begin{displaymath}
		\liminf_{N \rightarrow \infty} \, N^{\frac{1}{2}} \, \inf_{u \in \mathcal{U}}	\, \mathbb{E}\left[  \big| u( W_{t_1},W_{t_2}, \ldots, W_{t_N}, B_{t_1},B_{t_2}, \ldots, B_{t_N}) -X_T \big|\right]  >0,
	\end{displaymath}
	for $ \nu >1$, where $\mathcal{U}$ is the set of measurable functions from $\mathbb{R}^{2N}$ to $\mathbb{R}$.

	The strong approximation of the full Heston model, i.e.~of $S_T$ instead of $X_T$ carries  an additional burden, since the SDE for the asset price has superlinear coefficients and admits moment explosions, i.e. ${\mathbb E}(S_T^p)=\infty$ for certain parameter constellations and $p>1$, see e.g.~\cite{And-M}. The article \cite{CS}, where
	exponential integrability results for several Euler-type methods for the CIR process have been established, is dedicated to this problem.

	\subsection{Further results}\label{sec:further}
	
	Our analysis is  taylor-made for the $L^1$-approximation  and based on the Tanaka-Meyer formula combined with a clever control of the arising local time of the error process. We found this approach in \cite{DA}, where the approximation of SDEs with irregular drift and additive noise has been studied.
	For the $L^p$-approximation with $p \geq 1$ we could deduce the upper bound $ \min \{1,\nu \}/(2p)$ for the convergence order by 
	a standard application of the H\"older inequality. However, this bound is unlikely to be sharp, compare e.g.  \cite{BER} and \cite{CS2}, so we do not spell out this result in detail.

	\smallskip

	More importantly, our results can in particular be helpful for the Monte-Carlo pricing of (path-dependent) European options, since they allow to control the bias:

	\begin{proposition}\label{prop:up-bound_path_dependent} Let $\nu>1$, $\epsilon >0$ and $(\hat{x}_{t_{k}})_{k \in \{ 0, \ldots, N \} }$  as in Theorem \ref{thm:up-bound}. Moreover, let $G: \mathbb{R}^{[0,T] }\rightarrow [0, \infty)$ be a measurable mapping for which
	 there exists  $L_G>0$ such that
			$$ |G(y)-G(z)| \leq L_G \sup_{t \in [0,T]}|y_t-z_t| $$
			for all measurable $y,z \in  \mathbb{R}^{[0,T] }$.	Finally, set $\hat{x}^{pc,N}_t= \hat{x}_{\eta(t)}$ with $\eta(t)=\max\{k\in \{0,...,N\}:t_k\le t \} \Delta t$.
		Then we have
		\begin{displaymath}
			\begin{aligned}
				&  \lim_{N \rightarrow \infty} \,\,   N^{1/2-\epsilon} \, \big|  \mathbb{E} \left[   G(X) \right]-   \mathbb{E} \left[  G\left(\hat{x}^{pc,N}\right) \right] \big|  =0.
			\end{aligned}
		\end{displaymath}	
	\end{proposition}
	
	\smallskip
	
	Examples for $G$ include lookback-put options as 
	$$ G(X)= \left[K-\sup_{t \in [0,T]}  \exp(X_t) \right]^+$$
	or  arithmetic-asian-put options
	$$  G(X)=  \left[K-\frac{1}{T}  \int_0^T  \exp(X_t)  dt \right]^{+}$$ with $K>0$.  Note that for  $X=W$ and $G(x)=\sup_{t \in [0,T]} x_t$, we have the well known result
	$$   \lim_{N \rightarrow \infty} {N}^{1/2} \left( \mathbb{E}\left[\max_{t \in [0,T]} W_t \right] -   \mathbb{E}\left[\max_{k=0, \ldots, N} W_{t_k}\right] \right)=  \sqrt{\frac{T}{2 \pi}} \, \left| \zeta  ( \tfrac{1}{2} ) \right| ,$$
	see e.g. pages 884 -- 886 in \cite{pitman},  where $\zeta(\cdot)$ denotes the Riemann zeta function. So,  in Proposition \ref{prop:up-bound_path_dependent} we can not expect to obtain a better decay of the bias.

	\subsection{Notation and Outline}
	As already mentioned, we will work with an equidistant discretization
	\begin{displaymath}
		t_k = k \Delta t, \quad k=0, \ldots, N,
	\end{displaymath}
	with $\Delta t= T/ N$ and $N \in \mathbb{N}$. Furthermore, we define $n(t):=\max\{k\in \{0,...,N\}:t_k\le t \}$ and $\eta(t):=t_{n(t)}$. Constants whose values depend only on $T, x_0, v_0, \kappa, \theta, \sigma, \mu, \rho$ and the choice of $f_1,f_2,f_3$ will be denoted in the following by $C$, regardless of their value. Other dependencies will be denoted by subscripts, that is $C_{h,p}$ means e.g.  that this constant depends additionally on  a function $h$ and a parameter $p$. Moreover, the value of all these constants can change from line to line. Finally, we will work on  a  filtered probability space $(\Omega, \cF, (\cF_t)_{t\in[0,T]}, P)$, where the filtration satisfies the usual conditions, and (in-)equalities between random variables or random processes are understood $P$-a.s.~unless mentioned otherwise.
	
	The remainder of this article is organized as follows. We first show and collect some preliminary results in Section \ref{sec:prelim}. 
	The proofs of Theorem \ref{thm:up-bound} and Proposition \ref{prop:up-bound} are  carried out in Section \ref{sec:main} and Section \ref{sec:heston}, while the proof of  Proposition \ref{prop:up-bound_path_dependent} is also given in  Section \ref{sec:heston}. Finally, our  simulation study is presented in Section \ref{sec:simulation}.

	\section{Time-continuous extensions of the schemes and other preliminary results}
	\label{sec:prelim}
	In this section, we will present  the discretization schemes in detail that we are analyzing and a couple of preliminary results that are needed  to prove our main theorems. The first one is a well-known result for the CIR process (see e.g. Section 3 in \cite{DNS} and Theorem 3.1 in \cite{HK}). 
	\begin{lemma}\label{bounded}
		Let $p>-\nu$. Then we have
		\begin{displaymath}
			\sup\limits_{t\in[0,T]}\mathbb{E}\left[V^{p}_t\right]<\infty.
		\end{displaymath}
	\end{lemma}

	The next auxiliary result on the smoothness of the CIR process and the log-Heston-SDE is also well known:

	\begin{lemma}\label{bounded-2}
		Let $p \geq 1$. Then we have
		\begin{displaymath}
			\sup_{0 \leq s< t \leq  T } \mathbb{E} \left[  \frac{|V_t-V_s|^p}{|t-s|^{p/2}} \right] +  		\sup_{0 \leq s< t \leq  T } \mathbb{E} \left[ \frac{|X_t-X_s|^p}{|t-s|^{p/2}} \right]<\infty.
		\end{displaymath}
	\end{lemma}
	\begin{proof} Let $p\geq 2$. The case $p \in [1,2)$ then follows by an application of the Lyapunov inequality.
		
	 By Jensen's inequality, it holds that 	$\left|x+y\right|^p\le 2^{p-1}\left( |x|^p+|y|^p\right)$ and $\left|x+y+z\right|^p\le 3^{p-1}\left(|x|^p+|y|^p+|z|^p\right)$. Together with H\"older's inequality we have that
		\begin{align*}
			\sup_{0 \leq s< t \leq  T } \mathbb{E} \left[ 	 \frac{|V_t-V_s|^p}{|t-s|^{p/2}} \right] &= 	\sup_{0 \leq s< t \leq  T } \frac{1}{|t-s|^{p/2}}\mathbb{E} \left[  \left|\int_{s}^{t}\kappa(\theta-V_u)du+\sigma\int_{s}^{t}\sqrt{V_u}dW_u\right|^p \right]\\
			&\le C_p\left( \sup_{0 \leq s< t \leq  T } \frac{|t-s|^{p-1}}{|t-s|^{p/2}}\mathbb{E}\left[\int_{s}^{t}\left|\kappa(\theta-V_u)\right|^pdu\right]\right.\\
			&\qquad\qquad\left.+ \sup_{0 \leq s< t \leq  T } \frac{1}{|t-s|^{p/2}}\mathbb{E}\left[\left|\sigma\int_{s}^{t}\sqrt{V_u}dW_u\right|^p\right]\right).
		\end{align*}
	The first term is bounded by Lemma \ref{bounded}. For the second term, an application of the Burkholder-Davis-Gundy inequality, H\"older's inequality and Lemma \ref{bounded} gives the desired result. The log-price process can be treated analogously, since its integral equation does not depend on $X$.
	\end{proof}

	The following lemma gives us a bound for the  expected local time in zero of a semimartingale. It can be found in Lemma 5.1 in \cite{DA}.
	\begin{lemma}\label{lemmadeAngelis}
		For any $\delta\in(0,1)$ and any real-valued, continuous semimartingale $Y=(Y_t)_{t \in [0,T]}$, we have
		\begin{displaymath}
			\begin{aligned}
				\mathbb{E}\left[L^0_t(Y)\right]\le 4\delta&-2\mathbb{E}\left[\int_{0}^{t}\left(\mathbbm{1}_{\{Y_s\in(0,\delta)\}}+\mathbbm{1}_{\{Y_s \geq \delta\}}e^{1-\frac{Y_s}{\delta}}\right)dY_s\right]\\
				&+\frac{1}{\delta}\mathbb{E}\left[\int_{0}^{t}\mathbbm{1}_{\{Y_s>\delta\}}e^{1-\frac{Y_s}{\delta}}d\langle Y\rangle_s\right], \qquad t \in [0,T],
			\end{aligned}
		\end{displaymath}
	where $L^0(Y)=(L_t^0(Y))_{t\in [0,T]}$ denotes the local time  of $Y$ in zero.
	\end{lemma}
	
	The following statement can be verified by a simple computation.
	\begin{lemma}\label{squareroot}
		For $\lambda\in[0,1]$ and $x,y\ge0$, we have
		\begin{displaymath}
			\begin{aligned}
				\left|\sqrt{x}-\sqrt{y}\right|\le 
				x^{-\frac{1}{2}(1-\lambda)}\left|x-y\right|^{1-\frac{\lambda}{2}}.
			\end{aligned}
		\end{displaymath}
	\end{lemma}
	
We also require the well known Burkholder-Davis-Gundy inequalities, see e.g.
	Theorem 3.28 in Chapter III of \cite{KS}.
	\begin{proposition}\label{prop:martingalemoments} 
		Let $M=(M_t)_{t \in [0,T]}$ be a  continuous martingale and $\alpha > 0$. Then, there exist  constants $k_{\alpha},K_{\alpha}>0$ 	 such that 

		\begin{displaymath}
			k_{\alpha} \, \mathbb{E}\left[\langle M\rangle_t^{\alpha}\right] \leq 
			\mathbb{E}\left[ \sup_{u \in [0,t]} \left|M_u\right|^{2\alpha} \right] \leq K_{\alpha} \, \mathbb{E}\left[\langle M\rangle_t^{{\alpha}} \right], \qquad t \in [0,T].
		\end{displaymath}
		Here $\langle M\rangle_t$, $t \in [0,T]$, denotes the quadratic variation of $M$.
	\end{proposition}

Finally, we will also need Doob's maximal inequality, see e.g. Theorem 3.8 in Chapter I of \cite{KS}. 
		\begin{proposition}\label{prop:doob} 
		Let $M=(M_t)_{t \in [0,T]}$ be a  continuous martingale and $p >1$. Then,  it holds that
		\begin{displaymath}
		\mathbb{E}\left[ \sup_{u \in [0,t]} \left|M_u\right|^{p} \right] \leq  \left(\frac{p}{p-1} \right)^p \mathbb{E}|M_t|^p, \qquad t \in [0,T].
	\end{displaymath}
\end{proposition}

	\subsection{Euler schemes -- Case I}\label{prelim-schemes1}
	For the choice \eqref{cir-euler},  \eqref{cir-choices1}, \eqref{heston-euler}  the time-continuous extensions $\bar{v}=(\bar{v}_t)_{t \in [0,T]}$ and  $\hat{v}=(\hat{v}_t)_{t \in [0,T]}$ read as
	\begin{equation}\label{first}
		\begin{aligned}
			\bar{v}_t&=\bar{v}_{\eta(t)}+\int_{\eta(t)}^{t}\kappa(\theta - f_2(\bar{v}_{\eta(s)})) ds + \sigma \int_{\eta(t)}^{t}\sqrt{f_3(\bar{v}_{\eta(s)})}dW_s, \\
			\hat{v}_t &= f_3\left(\bar{v}_t\right),
		\end{aligned} \qquad t \in [0,T],
	\end{equation}
	with $f_2\in \{	{\tt id}, 	{\tt abs}, 	{\tt sym} \} $, $f_3 \in \{ 	{\tt abs}, 	{\tt sym} \}$ and $\bar{v}_0=v_0$.
	Note that $f_2$ and $f_3$ are globally Lipschitz continuous with Lipschitz constant $L=1$ and
	satisfy 
	\begin{equation} |x-f_i(y)| \leq |x-y|, \qquad x \geq 0, y \in \mathbb{R}, \quad i=2,3.\end{equation}
	Moreover note that
	\begin{equation} \label{eq_sqr}
		\sqrt{|f_i(x)|} \leq 1+|x|, \qquad x \in \mathbb{R}, \quad i=1,2,3. 
	\end{equation}
	
	The next lemma can be shown by some tedious but straightforward computations, since the coefficients of the Euler scheme are of linear growth. 

	\begin{lemma}\label{boundedscheme}
		Let $p \geq 1$. There exists  $C_p >0$ such that
		\begin{displaymath}
			\mathbb{E}\left[  \sup_{t\in [0,T]} |\bar{v}_t|^p \right]  + \sup_{0 \leq s< t \leq  T } \mathbb{E} \left[ \frac{ | \bar{v}_t - \bar{v}_s|^p}{|t-s|^{p/2}}  \right]  \leq C_p.
		\end{displaymath}
	\begin{proof}Let $p\geq 2$. The case $p \in [1,2)$ then follows by an application of the Lyapunov inequality.
		
		The boundedness of the second term follows from the boundedness of the first term by applying the same technique as in the proof of Lemma \ref{bounded-2}, where we use that Equation \eqref{eq_sqr} holds for $f_2$ and $f_3$.
		
		 For the first term, we will start with establishing
		\begin{equation}\label{bound_v}
			\sup_{t \in [0,T]}\mathbb{E}\left[|\bar{v}_t|^p\right]<\infty.
		\end{equation}
Let $\tau_n$ be the stopping time defined by $\tau_n:=\inf\{0<t<T;\bar{v}_t\ge n\}$ with $\inf\{\emptyset\}=T$. We have $|f_i(\bar{v}_{\eta(t)})|\le |\bar{v}_{\eta(t)}|$ for $i\in\{2,3\}$. Thus, it follows
	\begin{align*}
		\mathbb{E}\left[|\bar{v}_{t\wedge\tau_n}|^p\right]&\le \mathbb{E}\left[\left|v_0+\int_{0}^{t\wedge\tau_n}\kappa(\theta - f_2(\bar{v}_{\eta(s)})) ds + \sigma \int_{0}^{t\wedge\tau_n}\sqrt{f_3(\bar{v}_{\eta(s)})}dW_s\right|^p\right]\\
		&\le C_p\left(v_0^p+\mathbb{E}\left[\left|\int_{0}^{t\wedge\tau_n}\kappa(\theta - f_2(\bar{v}_{\eta(s)})) ds\right|^p\right]+\mathbb{E}\left[\left|\sigma \int_{0}^{t\wedge\tau_n}\sqrt{f_3(\bar{v}_{\eta(s)})}dW_s\right|^p\right]\right)\\
		&\le C_p\left(1+\mathbb{E}\left[\left|\int_{0}^{t\wedge\tau_n}\kappa(\theta - f_2(\bar{v}_{\eta(s)})) ds\right|^p\right]+\mathbb{E}\left[\left|\sigma^2 \int_{0}^{t\wedge\tau_n}(1+|\bar{v}_{\eta(s)}|)^2ds\right|^{p/2}\right]\right)\\
		&\le C_p\left(1+\mathbb{E}\left[\int_{0}^{t\wedge\tau_n}|\bar{v}_{\eta(s)}|^p ds\right]+\mathbb{E}\left[\left| \int_{0}^{t\wedge\tau_n}|\bar{v}_{\eta(s)}|^2ds\right|^{p/2}\right]\right)\\
		&\le C_p\left(1+\mathbb{E}\left[\int_{0}^{t\wedge\tau_n}|\bar{v}_{\eta(s)}|^p ds\right]\right)
	\end{align*}
	by applications of the H\"older and the Burkholder-Davis-Gundy inequality. Therefore, we have shown that
	\begin{align*}
		\mathbb{E}\left[|\bar{v}_{t\wedge\tau_n}|^p\right]\le C_p\left(1+\int_{0}^{t}\mathbb{E}\left[|\bar{v}_{\eta(s)\wedge \tau_n}|^p\right] ds\right)
	\end{align*}
and we consequently obtain
	\begin{align*}
\sup_{s \in [0,t]}	\mathbb{E}\left[|\bar{v}_{s \wedge\tau_n}|^p\right]\le C_p\left(1+\int_{0}^{t} \sup_{u \in [0,s]}\mathbb{E}\left[|\bar{v}_{u\wedge \tau_n}|^p\right] ds\right).
\end{align*}

The Gronwall inequality now yields
	\begin{align*}
		\sup_{t \in [0,T]}\mathbb{E}\left[|\bar{v}_{t\wedge\tau_n}|^p\right]\le C_p,
	\end{align*}
	where the constant $C_p>0$ does in particular not depend on $n$. Taking the limit $n\rightarrow\infty$, we obtain \eqref{bound_v}. Due to
	\begin{align*}
		\sup_{t\in [0,T]}|\bar{v}_t|^p\le C_p\left(1+\int_{0}^{T}\left|\bar{v}_{\eta(s)}\right|^pds+\sup_{t \in [0,T]}\left|\int_{0}^{t}\sigma\sqrt{f_3(\bar{v}_{\eta(s)})}dW_s\right|^{p}\right),
	\end{align*}
	the assertion now follows from the properties of $f_3$, an application of the Burkholder-Davis-Gundy inequality and Equation \eqref{bound_v}.
	\end{proof}	
	\end{lemma}

	\subsection{Euler schemes -- Case II}\label{prelim-schemes2} 
	For \eqref{cir-euler},  \eqref{cir-choices2}, \eqref{heston-euler} we obtain the
	Symmetrized Euler (SE) and the Euler with Absorption (AE). We can write the time-continuous extension $	\hat{v}^{sym}= 	(\hat{v}^{sym}_t)_{t \in [0,T]}$ of  (SE) as
	\begin{displaymath}
		\hat{v}^{sym}_t = \left|\hat{v}^{sym}_{\eta(t)}+\kappa\left(\theta-\hat{v}^{sym}_{\eta(t)}\right)(t-\eta(t))+\sigma\sqrt{\hat{v}^{sym}_{\eta(t)}}\left(W_t-W_{\eta(t)}\right)\right|
	\end{displaymath}
	on each interval $[t_k,t_{k+1})$ and the time-continuous extension $	\hat{v}^{abs}= 	(\hat{v}^{abs}_t)_{t \in [0,T]}$ of  (AE) as
	\begin{displaymath}
		\hat{v}^{abs}_t = \left(\hat{v}^{abs}_{\eta(t)}+\kappa\left(\theta-\hat{v}^{abs}_{\eta(t)}\right)(t-\eta(t))+\sigma\sqrt{\hat{v}^{abs}_{\eta(t)}}\left(W_t-W_{\eta(t)}\right)\right)^{+},
	\end{displaymath}
	respectively.
	Now, let  $\star \in \{sym, abs\}$. We define
	\begin{displaymath}
		z_t^{\star}:=\hat{v}_{\eta(t)}^{\star}+\kappa\big(\theta-\hat{v}_{\eta(t)}^{\star}\big)(t-\eta(t))+\sigma\sqrt{\hat{v}_{\eta(t)}^{\star}}\left(W_{t}-W_{\eta(t)}\right)
	\end{displaymath}
	and use the Tanaka-Meyer formula  for $\hat{v}^{sym}_t=\bigl|z_t^{sym}\bigr|$ and  for $\hat{v}^{abs}_t=\left(z_t^{abs}\right)^+$ to obtain
	\begin{equation*}
		\begin{aligned}
			\hat{v}^{sym}_t=\hat{v}_{\eta(t)}^{sym} & +\int_{\eta(t)}^{t} \sign\left(z_s^{sym}\right)\kappa\left(\theta-\hat{v}_{\eta(s)}^{sym}\right)ds+\sigma\int_{\eta(t)}^{t}\sign\left(z_s^{sym}\right)\sqrt{\hat{v}_{\eta(s)}^{sym}}dW_s  \\ & + L_t^0(z^{sym})-L_{\eta(t)}^0(z^{sym}),\qquad t \in [0,T],
		\end{aligned}
	\end{equation*}
	and
	\begin{equation*}
		\begin{aligned}
			\hat{v}^{abs}_t=\text{ }\hat{v}_{\eta(t)}^{abs} &+\int_{\eta(t)}^{t}\mathbbm{1}_{\{z_s^{abs} > 0\}}\kappa\left(\theta-\hat{v}_{\eta(s)}^{abs}\right)ds+\sigma\int_{\eta(t)}^{t}\mathbbm{1}_{\{z_s^{abs}> 0\}}\sqrt{\hat{v}_{\eta(s)}^{abs}}dW_s \\ & +\frac{1}{2}\left( L_t^0(z^{abs})-L_{\eta(t)}^0(z^{abs})\right),\qquad t \in [0,T].
		\end{aligned}
	\end{equation*}
	Here $L^0(z^{\star})=(L_t^0(z^{\star}))_{t\in [0,T]}$ is the local time  of $z^{\star}$ in $x=0$. For almost all $\omega \in \Omega$  the map $[0,T] \ni t\mapsto [L_t^0(z^{\star})](\omega) \in \mathbb{R}$ is continuous and non-decreasing with $L_0^0(z^{\star})=0$. See e.g.~Theorem 7.1 in Chapter III of  \cite{KS}.
	We can rewrite both schemes as
	\begin{equation}\label{helpsym}
		\begin{aligned}
			\hat{v}^{\star}_t=&\hat{v}_{\eta(t)}^{\star} +\int_{\eta(t)}^{t} \kappa\left(\theta-\hat{v}_{\eta(s)}^{\star}\right)ds+\sigma\int_{\eta(t)}^{t}\sqrt{\hat{v}_{\eta(s)}^{\star}}dW_s\\
			&-2c^{\star}\sigma\int_{\eta(t)}^{t}\mathbbm{1}_{\{z_s^{\star} \le 0\}} \sqrt{\hat{v}_{\eta(s)}^{\star}}dW_s-2c^{\star}\int_{\eta(t)}^{t}\mathbbm{1}_{\{z_s^{\star} \le 0\}} \kappa\left(\theta-\hat{v}_{\eta(s)}^{\star}\right)ds\\
			&+c^{\star}\left( L_t^0(z^{\star})-L_{\eta(t)}^0(z^{\star})\right), \qquad t \in [0,T],
		\end{aligned}
	\end{equation}
	with $c^{sym}=1$ and $c^{abs}=\frac{1}{2}$.

	\begin{lemma}\label{boundedscheme2} Let $\star \in \{sym,abs\}$ and $p \geq 1$. Then, there exists  $C_p >0$ such that
		\begin{displaymath}
			\mathbb{E}\left[  \sup_{t\in [0,T]} |\hat{v}_t^{\star}|^p \right] +	\sup_{t \in [0,T]}  \mathbb{E} \left[  \frac{ | \hat{v}^{\star}_{t} - \hat{v}^{\star}_{\eta(t)}|^p}{ \, |t-\eta(t)|^{p/2}}  \right] \leq C_p.
		\end{displaymath} 
	\end{lemma}
	\begin{proof}
		The finiteness of the first summand can be found in Lemma 2.1 in \cite{BO}  for the symmetrized Euler scheme and can obtained analogously for the absorbed Euler scheme. The finiteness of the second summand is established in Lemma 3.7 in
		\cite{mickel2021weak}.
	\end{proof}
	The next two lemmas are Propositions 3.6  and 3.9 from 	\cite{mickel2021weak}.

	\begin{lemma}\label{Zzero}
		For $\Delta t < \frac{1}{\kappa}$ and $\varepsilon\in(0,1/2]$ we have that
		\begin{equation}\label{problemma3.6}
			P(z_t^{\star} \leq 0) \leq c \left(\frac{\Delta t}{\varepsilon}\right)^{\nu(1-\varepsilon)}, \qquad t \in [0,T] \setminus \{t_0,t_1, \ldots, t_N\},
		\end{equation}
		for $\star \in \{sym, abs\}$,  where $c$ is given by  
			\begin{align*} \label{const_lemma3.7} c= \exp\left(\kappa\left(\nu T+\frac{2v_0}{\sigma^2}\right)\right)\left(\max\left\{1,\frac{\sigma^2\nu}{v_0e}\right\}\right)^{\nu}. \end{align*} 
	\end{lemma}

	\begin{lemma}\label{boundlocal}
		Let $\beta>0$,  $\varepsilon \in (0,1/2]$, $\Delta t\le\frac{1}{2\kappa}$ and $\star \in \{sym, abs\}$. Then, there exists a constant $C_{\beta}>0$  such that
		\begin{displaymath}
			\mathbb{E}\left[L_t^0\left(z^{\star}\right)-L_{\eta(t)}^0\left(z^{\star}\right)\right]\le C_\beta\Delta t \left(\frac{\Delta t}{\varepsilon}\right)^{\nu\frac{1-\varepsilon}{1+\beta}}, \qquad t \in [0,T].
		\end{displaymath}
	\end{lemma}

	The following Lemma gives a control of the non-martingale terms, which arise additionally in the expansion of SE and AE, i.e. in \eqref{helpsym}.

	\begin{lemma}\label{Rem-SE-AE} 	Let $\Delta t\le\frac{1}{2\kappa}$, $\varepsilon \in (0,1/2]$, $\beta>0$ and $\star \in \{sym, abs\}$. Moreover,
		let $g:\mathbb{R}^2 \rightarrow \mathbb{R}$ be measurable and  bounded, and let $h: \mathbb{R} \rightarrow \mathbb{R}$ be measurable and of linear growth. 
		Then there exist constants $C_{g,h,\beta}>0$ and  $C_{g, \beta} >0$ such that 
		\begin{displaymath}
			\sup_{t\in [0,T]} \mathbb{E} \left[ \left|  \int_0^t g(V_u,\hat{v}^{\star}_u) h(\hat{v}^{\star}_{\eta(u)}) \mathbbm{1}_{\{z_u^{\star} \le 0\}} du \right|\right]  \le C_{g,h,\beta} \left(\frac{\Delta t}{\varepsilon}\right)^{\nu\frac{1-\varepsilon}{1+\beta}}
		\end{displaymath}
		and
		\begin{displaymath}
			\sup_{t\in [0,T]} \mathbb{E} \left[ \left| \int_0^t  g(V_u,\hat{v}^{\star}_u) d L^0_u(z^{\star})\right|\right]  \le C_{g, \beta } \left(\frac{\Delta t}{\varepsilon}\right)^{\nu\frac{1-\varepsilon}{1+\beta}}.
		\end{displaymath}
	\end{lemma}
	
	\begin{proof}
		(a) We start with the second assertion. Note that the integral under consideration is a pathwise Riemann-Stieltjes integral,  since $L^0(z^{\star})$ is positive and non-decreasing with $L_0^0(z^{\star})=0$. With $\|g \|_{\infty}= \sup_{x,y \in\mathbb{R}} |g(x,y)|$
		we then have
		\begin{displaymath}
			- \|g \|_{\infty} L^0_T(z^{\star}) \leq \int_0^t  g(V_u,\hat{v}^{\star}_u) d L^0_u(z^{\star}) \leq \|g \|_{\infty} L^0_T(z^{\star}), \qquad t \in [0,T].
		\end{displaymath}
		It follows 
		\begin{displaymath}
			\sup_{t\in [0,T]} \mathbb{E} \left[ \left| \int_0^t  g(V_u,\hat{v}^{\star}_u) d L^0_u(z^{\star})\right| \right]  \leq \|g \|_{\infty} \sum_{k=0}^{N-1}  \mathbb{E} \left[
			L_{t_{k+1}}^0\left(z^{\star}\right)-L_{t_k}^0\left(z^{\star}\right)\right]
		\end{displaymath}
		and Lemma \ref{boundlocal} gives
		\begin{displaymath}
			\sup_{t\in [0,T]} \mathbb{E} \left[ \left| \int_0^t  g(V_u,\hat{v}^{\star}_u) d L^0_u(z^{\star})\right|\right]   \leq  C_{g,\beta}   \left(\frac{\Delta t}{\varepsilon}\right)^{\nu\frac{1-\varepsilon}{1+\beta}}.
		\end{displaymath}
		
		(b) For the first assertion note that
		\begin{displaymath}
			\begin{aligned}
				& \sup_{t\in [0,T]} \mathbb{E} \left[ \left|  \int_0^t g(V_u,\hat{v}^{\star}_u) h(\hat{v}^{\star}_{\eta(u)}) \mathbbm{1}_{\{z_u^{\star} \le 0\}} du \right|\right]   \\ & \qquad  \leq C_h \| g \|_{\infty} \int_0^T \mathbb{E} \left[  \left(1+ \sup_{t \in [0,T]} |\hat{v}^{\star}_t| \right) \mathbbm{1}_{\{z_u^{\star} \le 0\}} \right] du.
			\end{aligned}     
		\end{displaymath} 
		An application of H\"older's inequality together with Lemma \ref{boundedscheme2}  then yields
		\begin{displaymath}
			\sup_{t\in [0,T]} \mathbb{E} \left[ \left|  \int_0^t g(V_u,\hat{v}^{\star}_u) h(\hat{v}^{\star}_{\eta(u)})
			\mathbbm{1}_{\{z_u^{\star} \le 0\}} du \right|\right]  \leq C_{ g,h, \beta } \int_0^T (P(z_u^{\star} \le 0))^{\frac{1}{1+\beta}} du     
		\end{displaymath} 
		for all $\beta >0$. Lemma \ref{Zzero}  implies now that
		\begin{displaymath}
			\begin{aligned}
				&\sup_{t\in [0,T]} \mathbb{E} \left[ \left|  \int_0^t g(V_u,\hat{v}^{\star}_u) h(\hat{v}^{\star}_{\eta(u)}) \mathbbm{1}_{\{z_u^{\star} \le 0\}} du \right| \right]
				\quad \leq C_{g, h, {\beta} }  \left(\frac{\Delta t}{\varepsilon}\right)^{\nu\frac{1-\varepsilon}{1+\beta}}.
			\end{aligned} 
		\end{displaymath} 
		which finishes the proof.
	\end{proof}

	\subsection{The Euler scheme for the log-price process}
	The  time-continuous extension $ \hat{x}=(\hat{x}_t)_{t \in [0,T]}$ of the Euler scheme for the log-price process in the Heston model is given by
	\begin{equation}\label{eulerscheme}
		\begin{aligned}
			\hat{x}_{t}=\hat{x}_{\eta(t)}&+\left(\mu-\frac{1}{2}\hat{v}_{\eta(t)}\right)(t-\eta(t))+\rho\sqrt{\hat{v}_{\eta(t)}}\left(W_t-W_{\eta(t)}\right)\\
			&+\sqrt{1-\rho^2}\sqrt{\hat{v}_{\eta(t)}}\left(B_t-B_{\eta(t)}\right).
		\end{aligned}
	\end{equation}
	For $\hat{v}$, we can choose one of the previously introduced schemes for the CIR process. We have the same results concerning the moment stability and the local smoothness as before.

	\begin{lemma}\label{boundedscheme3} Let $p \geq 1$.
		For the Euler scheme \eqref{eulerscheme} together with the scheme  \eqref{first} or \eqref{helpsym}, there exists $C_p >0$ such that
		\begin{displaymath}
			\mathbb{E}\left[  \sup_{t\in [0,T]} |\hat{x}_t|^p \right]  +	\sup_{0 \leq s< t \leq  T }  \mathbb{E} \left[ \frac{ | \hat{x}_t - \hat{x}_s|^p}{|t-s|^{p/2}}  \right]  \leq C_p.
		\end{displaymath}
	\end{lemma}
	\begin{proof}
		
			Again, let $p\geq 2$. The case $p \in [1,2)$ then follows by an application of the Lyapunov inequality.
			
		Using the results from Lemma \ref{boundedscheme} and Lemma \ref{boundedscheme2}, both statements follow  by standard computations.
	For the first term, we use  the H\"older and the Burkholder-Davis-Gundy inequality, which leads to
		\begin{align*}
			\mathbb{E}\left[  \sup_{t\in [0,T]} |\hat{x}_t|^p \right]&=\mathbb{E}\left[  \sup_{t\in [0,T]} \left|x_0+\mu t-\frac{1}{2}\int_{0}^{t}\hat{v}_{\eta(s)}ds+\rho\int_{0}^{t}\sqrt{\hat{v}_{\eta(s)}}dW_s+\sqrt{1-\rho^2}\int_{0}^{t}\sqrt{\hat{v}_{\eta(s)}}dB_s\right|^p \right]\\
			&\le C_p \left(1+\mathbb{E}\left[  \sup_{t\in [0,T]}\left|\int_{0}^{t}\hat{v}_{\eta(s)}ds\right|^p\right]\right.\\
			&\left.\qquad \qquad +\mathbb{E}\left[  \sup_{t\in [0,T]}\left|\int_{0}^{t}\sqrt{\hat{v}_{\eta(s)}}dW_s\right|^p\right]
			+\mathbb{E}\left[  \sup_{t\in [0,T]}\left|\int_{0}^{t}\sqrt{\hat{v}_{\eta(s)}}dB_s\right|^p \right]\right)\\
			&\le C_p.
		\end{align*}
		The statement for the second term follows by the first statement and by analogous computations as in the proof of Lemma \ref{bounded-2}.
	
	\end{proof}

	\section{CIR process: convergence rates}
	\label{sec:main}
	In this section, we will prove the announced $L^1$-convergence rates
	for the Euler schemes presented in the introduction. 
	
	\subsection{Error analysis - Case I and $\mathbf{\nu>1}$}
	We will first look at the discretization from \eqref{first}  under the condition $\nu>1$.
	\begin{theorem}\label{thm:one}
		Let $\bar{v}$ be given by \eqref{first} and $\nu>1$.
		Then, for all $\varepsilon >0 $ there exists a constant $C_{\varepsilon}>0$ such that
		\begin{displaymath}
			\sup\limits_{t\in[0,T]}\mathbb{E}\left[ \left|V_t-\bar{v}_t\right| \right] \leq C_{\varepsilon} \left (\Delta t\right)^{\frac{1}{2}-\varepsilon}.
		\end{displaymath}
	\end{theorem}
	\begin{proof}
		Define $e=(e_t)_{t \in [0,T]}$ by $e_t=V_t-\bar{v}_t$. 
		
		(i) The Tanaka-Meyer formula, see e.g.~Equation 7.9 in Chapter III in \cite{KS}, yields
		\begin{equation}\label{prooffirststep}
			\begin{aligned}
				\mathbb{E}\left[ \left|e_t\right|\right]  =& \mathbb{E}\left[ \int_{0}^{t}\sign(e_u)de_u\right]+\mathbb{E}\left[L^0_t(e)\right]\\
				=&\mathbb{E}\left[\int_{0}^{t}\sign(e_u)\left(-\kappa\left(V_u-f_2(\bar{v}_{\eta(u)})\right)\right)du\right] \\ & \quad + \mathbb{E}\left[\int_{0}^{t}\sign(e_u)\sigma\left(\sqrt{V_u}-\sqrt{f_3(\bar{v}_{\eta(u)})}\right)dW_u\right]
				+ \mathbb{E} \left[L^0_t(e) \right].
			\end{aligned} 
		\end{equation}
		We have
		\begin{displaymath}
			\mathbb{E}\left[\int_{0}^{t}\sign(e_u)\sigma\left(\sqrt{V_u}-\sqrt{f_3(\bar{v}_{\eta(u)})}\right)dW_u\right]=0
		\end{displaymath}
		due to Lemma \ref{bounded}, Lemma \ref{boundedscheme}, the properties of $f_3$ and the martingale property of the It\=o integral.
		Looking at the first term, we have
		\begin{displaymath}
			\begin{aligned}
				&\mathbb{E}\left[\int_{0}^{t}\sign(e_u)\left(-\kappa\left(V_u-f_2(\bar{v}_{\eta(u)})\right)\right)du\right] \\ &=- \kappa \, \mathbb{E}\left[\int_{0}^{t}\sign(e_u)\left(V_u-f_2(\bar{v}_{u}) \right)du\right]
			 -\kappa \,  \mathbb{E}\left[\int_{0}^{t}\sign(e_u)\left(f_2(\bar{v}_{u})-f_2(\bar{v}_{\eta(u)})\right)du\right]\\
				&  \leq  \kappa \int_{0}^{t} \mathbb{E}\left[ |e_u|\right]  du+ \kappa \int_{0}^{t} \mathbb{E}\left[ |\bar{v}_u-\bar{v}_{\eta(u)}|\right]  du \\
				&  \le C(\Delta t)^{\frac{1}{2}} +  \kappa \int_{0}^{t} \mathbb{E}\left[ |e_u| \right] du
			\end{aligned}
		\end{displaymath}
		due to Lemma \ref{boundedscheme} and $|x-f_2(y)| \leq |x-y|$ for $x \geq 0, y \in \mathbb{R}$ as well as $|f_2(x)-f_2(y)| \leq |x-y|$ for  $x,y \in \mathbb{R}$.
		Therefore, we obtain
		\begin{equation}\label{cir-tanaka-1}
			\sup_{u \in [0,t]} \mathbb{E}\left[ \left|e_u\right|\right]  \leq  C(\Delta t)^{\frac{1}{2}} +  \kappa \int_{0}^{t}  	\sup_{v \in [0,u]} \mathbb{E}\left[ |e_v|\right]  du + \mathbb{E}\left[L^0_t(e)\right].
		\end{equation}
		
		(ii)	With Lemma \ref{lemmadeAngelis} we can derive a bound for the expected local time in 0 of $e$. Let $\delta\in(0,1)$, then
		\begin{equation}\label{local}
			\begin{aligned}
				\mathbb{E}\left[L^0_t(e)\right]\le 4\delta&-2\mathbb{E}\left[\int_{0}^{t}\left(\mathbbm{1}_{\{e_s\in(0,\delta)\}}+\mathbbm{1}_{\{e_s \geq \delta\}}e^{1-\frac{e_s}{\delta}}\right)de_s\right]\\
				&+\frac{1}{\delta}\mathbb{E}\left[\int_{0}^{t}\mathbbm{1}_{\{e_s>\delta\}}e^{1-\frac{e_s}{\delta}}d\langle e\rangle_s\right].
			\end{aligned}
		\end{equation}
		We define $Y_s:=\mathbbm{1}_{\{e_s\in(0,\delta)\}}+\mathbbm{1}_{\{e_s  \geq \delta\}}e^{1-\frac{e_s}{\delta}}$ and look at the second term of \eqref{local}, i.e. at
		\begin{displaymath}
			\mathbb{E}\left[\int_{0}^{t}Y_sde_s\right]=- \kappa \,  \mathbb{E}\left[\int_{0}^{t}Y_s\left(V_s-f_2(\bar{v}_{\eta(s)})\right) ds\right],
		\end{displaymath}
		where we already used the martingale property of the It\=o integral. Since $0 \leq Y_s \leq 1$, we obtain proceeding as above  that 
		\begin{equation}
			\begin{aligned} \label{LT-3}
				\left| 	\mathbb{E}\left[\int_{0}^{t}Y_sde_s\right] \right| & \leq  \kappa	\int_{0}^{t} 	\mathbb{E}\left[  |V_s-\bar{v}_{\eta(s)}|\right]  ds  \\ & \leq 
				C(\Delta t)^{\frac{1}{2}} +  \kappa \int_{0}^{t} \mathbb{E}\left[ |e_u| \right] du.
			\end{aligned}
		\end{equation}
		The third term of \eqref{local} can be bounded as follows with  Lemma \ref{squareroot} (using $\lambda =0$) and the properties of $f_3$:
		\begin{displaymath}
			\begin{aligned}
				\frac{1}{\delta}\mathbb{E}\left[\int_{0}^{t}\mathbbm{1}_{\{e_s>\delta\}}e^{1-\frac{e_s}{\delta}}d\langle e\rangle_s\right]&=\frac{1}{\delta}\mathbb{E}\left[\int_{0}^{t}\mathbbm{1}_{\{e_s>\delta\}}e^{1-\frac{e_s}{\delta}}\sigma^2\left(\sqrt{V_s}-\sqrt{f_3(\bar{v}_{\eta(s)}})\right)^2ds\right]\\
				&\le \frac{1}{\delta}\mathbb{E}\left[\int_{0}^{t}\mathbbm{1}_{\{e_s>\delta\}}e^{1-\frac{e_s}{\delta}}\sigma^2 \frac{\left|V_s-f_3(\bar{v}_{\eta(s)})\right|^{2}}{V_s}ds\right] \\
				&\le \frac{1}{\delta}\mathbb{E}\left[\int_{0}^{t}\mathbbm{1}_{\{e_s>\delta\}}e^{1-\frac{e_s}{\delta}}\sigma^2 \frac{\left|V_s-\bar{v}_{\eta(s)}\right|^{2}}{V_s}ds\right].
			\end{aligned}
		\end{displaymath}
		With Lemma \ref{boundedscheme}, Lemma \ref{bounded}, the Minkowski inequality and H\"older's inequality it follows
		\begin{equation} \label{LT-2}
			\begin{aligned}
				\frac{1}{\delta}\mathbb{E}\left[\int_{0}^{t}\mathbbm{1}_{\{e_s>\delta\}}e^{1-\frac{e_s}{\delta}}d\langle e\rangle_s\right]\le& \frac{C}{\delta}\mathbb{E}\left[\int_{0}^{t}\mathbbm{1}_{\{e_s>\delta\}}e^{1-\frac{e_s}{\delta}}\ \frac{\left|V_s-\bar{v}_{s}\right|^{2}}{{V_s}}ds\right]\\
				&+\frac{C}{\delta}\mathbb{E}\left[\int_{0}^{t}\mathbbm{1}_{\{e_s>\delta\}}e^{1-\frac{e_s}{\delta}} \frac{\left|\bar{v}_s-\bar{v}_{\eta(s)}\right|^{2}}{V_s}ds\right]\\
				\le&\frac{C}{\delta}\mathbb{E}\left[\int_{0}^{t}\mathbbm{1}_{\{e_s>\delta\}}e^{1-\frac{e_s}{\delta}}\ \frac{\left|e_s\right|^{2}}{V_s}ds\right]\\
				&+C \frac{\Delta t}{\delta} \int_{0}^{t} \left( \mathbb{E}\left[\frac{1}{V_s^{(1+\nu)/2}}\right] \right)^{\frac{2}{1+\nu}}ds		
				\\
				\le&\frac{C}{\delta}\mathbb{E}\left[\int_{0}^{t}\mathbbm{1}_{\{e_s>\delta\}}e^{1-\frac{e_s}{\delta}} \frac{\left|e_s\right|^{2}}{V_s}ds\right]\\
				&+C \left(\frac{\Delta t}{\delta}\right).
			\end{aligned}
		\end{equation}
		Now let $\alpha\in(0,1)$. Since
		\begin{displaymath}
			\sup_{s \in [0,T]} \mathbb{E}\left[ |e_s|^p\right]  \leq C_p
		\end{displaymath}
		for all $p \geq 1$ due to Lemma \ref{bounded} and Lemma \ref{boundedscheme}, we have that
		\begin{equation}
			\begin{aligned}
				& \frac{1}{\delta}\mathbb{E}\left[\int_{0}^{t}\mathbbm{1}_{\{e_s>\delta\}}e^{1-\frac{e_s}{\delta}} \frac{\left|e_s\right|^{2}}{V_s}ds\right] \\ & \quad =\frac{1}{\delta}\mathbb{E}\left[\int_{0}^{t}\mathbbm{1}_{\{e_s\in(\delta,\delta^{\alpha})\}}e^{1-\frac{e_s}{\delta}} \frac{\left|e_s\right|^{2}}{V_s}ds\right] \\ & \qquad \quad  +\frac{1}{\delta}\mathbb{E}\left[\int_{0}^{t}\mathbbm{1}_{\{e_s\ge\delta^{\alpha}\}}e^{1-\frac{e_s}{\delta}} \frac{\left|e_s\right|^2}{V_s}ds\right]\\
				& \quad \le \delta^{2\alpha-1}\int_{0}^{t}\mathbb{E}\left[\frac{1}{V_s}\right]ds+C \frac{e^{-\delta^{\alpha-1}}}{\delta}\int_{0}^{t} \left( \mathbb{E}\left[\frac{1}{V_s^{(1+\nu)/2}}\right] \right)^{\frac{2}{1+\nu}}ds\\
				& \quad \le C_{\alpha}\delta^{2\alpha-1} \label{LT-1}
			\end{aligned}
		\end{equation}
		by another application of H\"older's inequality and $ \limsup_{\delta \rightarrow 0} \frac{e^{-\delta^{\alpha-1}}}{\delta^{2\alpha}} =0$.
		
		Summarizing \eqref{local} -- \eqref{LT-1} we have shown that
		\begin{equation} \label{LT-0}
			\mathbb{E}\left[L^0_t(e)\right] \leq 4 \delta+	C(\Delta t)^{\frac{1}{2}} +  \kappa \int_{0}^{t} \mathbb{E}\left[ |e_u|\right]  du +  C_{\alpha} \delta^{2\alpha-1}+ C \frac{\Delta t}{\delta}.
		\end{equation}
		(iii) Setting $\delta = (\Delta t)^{1/2}$ gives
		\begin{equation} \begin{aligned}
				\mathbb{E}\left[L^0_t(e)\right] & \leq    \kappa \int_{0}^{t} \mathbb{E}\left[ |e_u| \right] du +  C_{\alpha} (\Delta t)^{\alpha-1/2} \\ & \leq    \kappa \int_{0}^{t} \sup_{v \in [0,u]}\mathbb{E}\left[   |e_v|\right]  du 
				+  C_{\alpha} (\Delta t)^{\alpha-1/2} . \end{aligned}
		\end{equation}	 
		Combining this with \eqref{cir-tanaka-1} yields
		\begin{displaymath}
			\sup_{u \in [0,t]}	\mathbb{E}\left[ \left|e_u\right|\right]  \leq   C_{\alpha} (\Delta t)^{\alpha-1/2} +  2 \kappa \int_{0}^{t} \sup_{v \in [0,u]}\mathbb{E}\left[   |e_v|\right]  du 
		\end{displaymath}
		and the assertion follows  by choosing $\alpha=1-\varepsilon$ and an application of Gronwall's lemma.
		
	\end{proof}

	\subsection{Error analysis - Case II and $\mathbf{\nu >1}$}
	
	In this section, we will analyze both schemes from \eqref{helpsym} under the condition $\nu>1$.
	
	\begin{theorem}\label{thm:two}
		Let $\hat{v}$ be given by \eqref{helpsym} and $\nu>1$.
		Then, for all $\varepsilon >0$ there exists a constant $C_{\varepsilon}>0$ such that
		\begin{displaymath}
			\sup\limits_{t\in[0,T]}\mathbb{E}\left[ \left|V_t-\hat{v}_t\right|\right] \leq  C_{\varepsilon} \left( \Delta t\right)^{\frac{1}{2}-\varepsilon}.
		\end{displaymath}
	\end{theorem}
	\begin{proof}
		The proof is very similar to the proof of Theorem \ref{thm:one}. Differences are only due to the additional terms in the expansion of the schemes and we will give the required additional steps in the following.  For $e_t=V_t-\hat{v}_t$ we have
		\begin{equation}\label{helpsym-error}
			\begin{aligned}
				e_t=& - \int_{0}^{t} \kappa\left(V_s-\hat{v}_{\eta(s)}^{\star}\right)ds+\sigma\int_{0}^{t} \left( \sqrt{V_s} -\sqrt{\hat{v}_{\eta(s)}^{\star}} \right) dW_s\\
				&+2c^{\star}\sigma\int_{0}^{t}\mathbbm{1}_{\{z_s^{\star} \le 0\}} \sqrt{\hat{v}_{\eta(s)}^{\star}}dW_s+2c^{\star}\int_{0}^{t}\mathbbm{1}_{\{z_s^{\star} \le 0\}} \kappa\left(\theta-\hat{v}_{\eta(s)}^{\star}\right)ds\\
				&-c^{\star} L_t^0(z^{\star}).
			\end{aligned}
		\end{equation}

		(i) The Tanaka-Meyer formula yields
		\begin{displaymath}
			\begin{aligned}
				\mathbb{E}\left[ \left|e_t\right|\right]  =& \mathbb{E}\left[ \int_{0}^{t}\sign(e_u)de_u\right]+\mathbb{E}\left[L^0_t(e)\right]\\
				=& -\kappa \mathbb{E}\left[\int_{0}^{t}\sign(e_u)(V_u-\hat{v}^{\star}_{\eta(u)})du\right] \\ & +\mathbb{E}\left[\sigma \int_{0}^{t}\sign(e_u)\left(\sqrt{V_u}-\sqrt{\hat{v}^{\star}_{\eta(u))}}\right)dW_u\right]
				+ \mathbb{E} \left[L^0_t(e) \right]
				\\ & 	+	\mathbb{E}\left[2c^{\star}\sigma\int_{0}^{t}\sign(e_u)\mathbbm{1}_{\{z_u^{\star} \le 0\}} \sqrt{\hat{v}_{\eta(u)}^{\star}}dW_u\right]
				\\ &   	+ 	\mathbb{E}\left[2c^{\star}\int_{0}^{t}\sign(e_u)\mathbbm{1}_{\{z_u^{\star} \le 0\}} \kappa\left(\theta-\hat{v}_{\eta(u)}^{\star}\right)du-c^{\star}\int_{0}^{t}\sign(e_u)dL^0_u(z^{\star})\right].
			\end{aligned} 
		\end{displaymath}	 
		However,  Lemma \ref{boundedscheme2} and the martingale property of It\=o integrals imply
		\begin{displaymath}
			\mathbb{E}\left[2c^{\star}\sigma\int_{0}^{t}\sign(e_u)\mathbbm{1}_{\{z_u^{\star} \le 0\}} \sqrt{\hat{v}_{\eta(u)}^{\star}}dW_u\right]=0
		\end{displaymath}
		and Lemma \ref{Rem-SE-AE} gives 
		\begin{displaymath}
			\left| 	\mathbb{E}\left[\int_{0}^{t}\sign(e_u) \left(  \mathbbm{1}_{\{z_u^{\star} \le 0\}} 2 \kappa(\theta-\hat{v}_{\eta(u)}^{\star})du-dL^0_u(z^{\star}) \right) \right]	\right| \leq C_{\varepsilon}\left(\Delta t\right)^{\nu\frac{1-\varepsilon}{1+\varepsilon}}
		\end{displaymath}
	by choosing $\beta=\varepsilon$.
		Thus, we have
		\begin{equation}  \label{cir-tanaka-2}
			\sup_{u \in [0,t]}	\mathbb{E}\left[ \left|e_u\right|\right]  \leq  C(\Delta t)^{\frac{1}{2}} + \kappa \int_{0}^{t}  \sup_{v \in [0,u]}\mathbb{E}\left[ |e_v|\right]  du + \mathbb{E}\left[L^0_t(e)\right]
		\end{equation}
		as in the first step of the previous proof by choosing $\varepsilon$ appropriately and exploiting that $\nu>1$.
		
		(ii) In the same way Lemma \ref{Rem-SE-AE} and Lemma \ref{boundedscheme2} also yield 
		\begin{displaymath}
			\begin{aligned} 
				\left| 	\mathbb{E}\left[\int_{0}^{t}Y_sde_s\right] \right| & \leq 
				C(\Delta t)^{\frac{1}{2}} + \kappa \int_{0}^{t} \mathbb{E}\left[ |e_u|\right]  du.
			\end{aligned}
		\end{displaymath}
		Finally, we have 
		\begin{displaymath}
			\begin{aligned}
				\langle e \rangle_t = \sigma^2 \int_0^t \left( \sqrt{V_s} - \sqrt{\hat{v}^{\star}_{\eta(s)}} +  2c^{\star} \mathbbm{1}_{\{z_s^{\star} \le 0\}} \sqrt{\hat{v}_{\eta(s)}^{\star}} \right)^2 ds.
			\end{aligned}
		\end{displaymath}
	But again Lemma \ref{Rem-SE-AE} with $\beta=\varepsilon$ gives that
		\begin{displaymath} 
			\begin{aligned}
				\frac{1}{\delta}\mathbb{E}\left[\int_{0}^{t}\mathbbm{1}_{\{e_s>\delta\}}e^{1-\frac{e_s}{\delta}}d\langle e\rangle_s\right]\le&
				\frac{2}{\delta}\mathbb{E}\left[\int_{0}^{t}\mathbbm{1}_{\{e_s>\delta\}}e^{1-\frac{e_s}{\delta}}
				\left( \sqrt{V}_s - \sqrt{\hat{v}_{\eta(s)}}\right)^2 \right] ds + C_{\varepsilon} \frac{(\Delta t)^{\nu \frac{1-\varepsilon}{1+\varepsilon}}}{\delta}	
			\end{aligned}
		\end{displaymath}
		and proceeding as in the previous proof we obtain that
		\begin{displaymath}
			\begin{aligned}
				\mathbb{E}\left[L^0_t(e)\right] \leq 4 \delta+	C(\Delta t)^{\frac{1}{2}} & +  \kappa \int_{0}^{t} \mathbb{E}\left[ |e_u|\right]  du +  C_{\alpha} \delta^{2\alpha-1} +C \frac{\Delta t}{\delta} +  C_{\varepsilon} \frac{(\Delta t)^{\nu \frac{1-\varepsilon}{1+\varepsilon}}}{\delta}	.
			\end{aligned}
		\end{displaymath} 
		(iii) Setting $\delta = (\Delta t)^{1/2}$ and using $\nu>1$ now yield
		\begin{displaymath}
			\mathbb{E}\left[L^0_t(e)\right] \leq    \kappa \int_{0}^{t} \mathbb{E}\left[ |e_u| \right] du +  C_{\alpha} (\Delta t)^{\alpha-1/2} + C_{\varepsilon}(\Delta t)^{\frac{1-\varepsilon}{1+\varepsilon}-1/2} .
		\end{displaymath}	 
		Combining this with \eqref{cir-tanaka-2} gives
		\begin{displaymath}
			\sup_{u \in [0,t]}	\mathbb{E}\left[ \left|e_u\right|\right]  \leq    C_{\alpha} (\Delta t)^{\alpha-1/2} + C_{\varepsilon}(\Delta t)^{\frac{1-\varepsilon}{1+\varepsilon}-1/2} +  2 \kappa \int_{0}^{t} \sup_{v \in [0,u]}\mathbb{E}\left[   |e_v|\right]  du. 
		\end{displaymath}
		Choosing $\alpha=1-\varepsilon$ and 
		observing that $\frac{1}{2} - \varepsilon  \geq \frac{1-\varepsilon}{1+\varepsilon}-\frac{1}{2} $ for all $\varepsilon \in (0,1)$
		an application of Gronwall's lemma yields then
		\begin{displaymath}
			\sup_{u \in [0,T]}	\mathbb{E}\left[ \left|e_u\right|\right]  \leq     C_{\varepsilon}(\Delta t)^{\frac{1-\varepsilon}{1+\varepsilon}-1/2} 
		\end{displaymath}
		and the assertion follows by choosing $\varepsilon$ appropriately.
		
	\end{proof}

	\subsection{Error analysis - Case I and $\mathbf{\nu \leq 1}$}
	Now we will study again  the discretization from \eqref{first}   but under the condition $\nu\leq 1$.
	\begin{proposition}\label{prop:one-proof-a}
		Let $\bar{v}$ be given by \eqref{first} and $\nu \leq 1$.
		Then, for all $\varepsilon >0 $ there exists a constant $C_{\varepsilon}>0$ such that
		\begin{displaymath}
			\sup\limits_{t\in[0,T]}\mathbb{E}\left[ \left|V_t-\bar{v}_t\right| \right] \leq C_{\varepsilon} \left (\Delta t\right)^{\frac{\nu}{2}-\varepsilon}.
		\end{displaymath}
	\end{proposition}
	\begin{proof}
		Define again $e=(e_t)_{t \in [0,T]}$ by $e_t=V_t-\bar{v}_t$.
		
		(i) Proceeding as in the proof of Theorem \ref{thm:one} we obtain
		\begin{equation}\label{cir-tanaka-3}
			\sup_{u \in [0,t]} \mathbb{E}\left[ \left|e_u\right|\right]  \leq  C(\Delta t)^{\frac{1}{2}} +  \kappa \int_{0}^{t}  	\sup_{v \in [0,u]} \mathbb{E}\left[ |e_v|\right]  du + \mathbb{E}\left[L^0_t(e)\right]
		\end{equation}
		and
		\begin{equation}\label{local-3}
			\begin{aligned}
				\mathbb{E}\left[L^0_t(e)\right]\le 4\delta& + 	C(\Delta t)^{\frac{1}{2}} +  \kappa \int_{0}^{t} \mathbb{E}\left[ |e_u| \right] du
				+\frac{1}{\delta}\mathbb{E}\left[\int_{0}^{t}\mathbbm{1}_{\{e_s>\delta\}}e^{1-\frac{e_s}{\delta}}d\langle e\rangle_s\right]
			\end{aligned}
		\end{equation}
		with 
		\begin{equation*}
			\langle e \rangle_t = \sigma^2 \int_0^t \left( \sqrt{V_s} - \sqrt{f_3(\bar{v}_{\eta(s)})}\right)^2 ds.
		\end{equation*}
		(ii)  For the remaining term in  \eqref{local-3} we apply Lemma \ref{squareroot} with $\lambda=1-\nu(1-\zeta)$ for  $\zeta \in (0,1)$  and Lemma \ref{bounded}, Lemma \ref{boundedscheme}, H\"older's and Minkowski's inequality  as well as the properties of $f_3$ to obtain
		\begin{equation} 
			\begin{aligned}
				& 	\frac{1}{\delta}\mathbb{E}\left[\int_{0}^{t}\mathbbm{1}_{\{e_s>\delta\}}e^{1-\frac{e_s}{\delta}}d\langle e\rangle_s\right] \\ & \quad \le \frac{C_{\zeta}}{\delta}\mathbb{E}\left[\int_{0}^{t}\mathbbm{1}_{\{e_s>\delta\}}e^{1-\frac{e_s}{\delta}}\ \frac{\left|V_s-\bar{v}_{s}\right|^{1+\nu(1-\zeta)}}{{V_s}^{\nu(1-\zeta)}}ds\right]\\
				& \qquad +\frac{C_{\zeta}}{\delta}\mathbb{E}\left[\int_{0}^{t}\mathbbm{1}_{\{e_s>\delta\}}e^{1-\frac{e_s}{\delta}} \frac{\left|\bar{v}_s-\bar{v}_{\eta(s)}\right|^{1+\nu(1-\zeta)}}{V_s^{\nu(1-\zeta)}}ds\right]\\
				& \quad \le \frac{C_{\zeta}}{\delta}\mathbb{E}\left[\int_{0}^{t}\mathbbm{1}_{\{e_s>\delta\}}e^{1-\frac{e_s}{\delta}}\ \frac{\left|e_s\right|^{1+\nu(1-\zeta)}}{V_s^{\nu(1-\zeta)}}ds\right]\\
				& \qquad +C_{\zeta} \frac{\left(\Delta t\right)^{ (1+\nu(1-\zeta))/2}}{\delta}  \label{LT-2-b} \int_{0}^{t} \left( \mathbb{E}\left[\frac{1}{V_s^{\nu(1-\zeta^2)}}\right] \right)^{\frac{1}{1+\zeta}}ds \\
				& \quad \leq \frac{C_{\zeta}}{\delta}\mathbb{E}\left[\int_{0}^{t}\mathbbm{1}_{\{e_s>\delta\}}e^{1-\frac{e_s}{\delta}} \frac{\left|e_s\right|^{1+\nu(1-\zeta)}}{V_s^{\nu(1-\zeta)}}ds\right]\\
				& \qquad +C_{\zeta}\left(\frac{\left(\Delta t\right)^{ (1+\nu(1-\zeta))/2}}{\delta}\right). 
			\end{aligned}
		\end{equation}
		Now let again $\alpha\in(0,1)$. Since
		\begin{displaymath}
			\sup_{s \in [0,T]} \mathbb{E}\left[ |e_s|^p\right]  \leq C_p
		\end{displaymath}
		for all $p \geq 1$ due to Lemma \ref{bounded} and Lemma \ref{boundedscheme}, we have that
		\begin{equation*}
			\begin{aligned}
				& \frac{1}{\delta}\mathbb{E}\left[\int_{0}^{t}\mathbbm{1}_{\{e_s>\delta\}}e^{1-\frac{e_s}{\delta}} \frac{\left|e_s\right|^{1+\nu(1-\zeta)}}{V_s^{\nu(1-\zeta)}}ds\right] \\ & \quad =\frac{1}{\delta}\mathbb{E}\left[\int_{0}^{t}\mathbbm{1}_{\{e_s\in(\delta,\delta^{\alpha})\}}e^{1-\frac{e_s}{\delta}} \frac{\left|e_s\right|^{1+\nu(1-\zeta)}}{V_s^{\nu(1-\zeta)}}ds\right] \\ & \qquad \quad  +\frac{1}{\delta}\mathbb{E}\left[\int_{0}^{t}\mathbbm{1}_{\{e_s\ge\delta^{\alpha}\}}e^{1-\frac{e_s}{\delta}} \frac{\left|e_s\right|^{{1+\nu(1-\zeta)}}}{V_s^{\nu(1-\zeta)}}ds\right]\\
				& \quad \le \delta^{({1+\nu(1-\zeta)})\alpha-1}\int_{0}^{t}\mathbb{E}\left[\frac{1}{V_s^{\nu(1-\zeta)}}\right]ds \\ & \qquad  \quad +C \frac{e^{-\delta^{\alpha-1}}}{\delta}\int_{0}^{t} \left( \mathbb{E}\left[\frac{1}{V_s^{\nu(1-\zeta^2)}}\right] \right)^{\frac{1}{1+\zeta}}ds\\
				& \quad \le C_{\zeta,\alpha}\delta^{(1+\nu(1-\zeta))\alpha-1} 
			\end{aligned}
		\end{equation*}
		by another application of H\"older's inequality, Lemma \ref{bounded} and $ \limsup_{\delta \rightarrow 0} \frac{e^{-\delta^{\alpha-1}}}{\delta^{2\alpha}} =0$.
		
		Summarizing the previous steps we have shown that
		\begin{equation*} \begin{aligned} 
				\mathbb{E}\left[L^0_t(e)\right] \leq 4 \delta & +C_{\zeta}\frac{\left(\Delta t\right)^{ (1+\nu(1-\zeta))/2}}{\delta}  +   C_{\zeta,\alpha}\delta^{(1+\nu(1-\zeta))\alpha-1} \\ & +   \kappa \int_{0}^{t} \mathbb{E}\left[ |e_u|\right]  du +  C (\Delta t)^{1/2}.
			\end{aligned}
		\end{equation*}
		Setting $\delta = (\Delta t)^{1/2}$ and $\alpha=1-\zeta$ gives
		\begin{displaymath}
			\mathbb{E}\left[L^0_t(e)\right] \leq    \kappa \int_{0}^{t} \sup_{v \in [0,u]}\mathbb{E}\left[ |e_v| \right] du + C_{\zeta}\left(\Delta t\right)^{ \nu(1-\zeta)/2} +
			C_{\zeta} (\Delta t)^{( \nu(1-\zeta)^2 -\zeta)/ 2}.
		\end{displaymath}	 
		Combining this with \eqref{cir-tanaka-3} yields
		\begin{displaymath}
			\sup_{u \in [0,t]}	\mathbb{E}\left[ \left|e_u\right|\right]  \leq    C_{\zeta} (\Delta t)^{( \nu(1-\zeta)^2 -\zeta)/ 2} +  2 \kappa \int_{0}^{t} \sup_{v \in [0,u]}\mathbb{E}\left[   |e_v|\right]  du 
		\end{displaymath}
		and the assertion follows  by choosing $\zeta$ sufficiently small and an application of Gronwall's lemma.
		
	\end{proof}

	\begin{rmk}
		We are not able to establish the analogous result to Proposition \ref{prop:one-proof-a} for Case II of the Euler schemes, since we have in that case
		\begin{equation*}
		\frac{1}{\delta}	\langle e \rangle_t = 	\frac{\sigma^2}{\delta}  \int_0^t \left( \sqrt{V_s} - \sqrt{\hat{v}^{\star}_{\eta(s)}} +  2c^{\star} \mathbbm{1}_{\{z_s^{\star} \le 0\}} \sqrt{\hat{v}_{\eta(s)}^{\star}} \right)^2 ds
		\end{equation*}
	instead of
	\begin{equation*}
		\frac{1}{\delta}	\langle e \rangle_t =  	\frac{\sigma^2}{\delta} \int_0^t \left( \sqrt{V_s} - \sqrt{f_3(\bar{v}_{\eta(s)})}\right)^2 ds.
	\end{equation*}
	The additional term gives a contribution of order $ 	\frac{1}{\delta} \Delta t^{\nu \frac{1-\varepsilon}{1+\varepsilon} }$, which will lead to a worse error bound than the one given in Proposition \ref{prop:one-proof-a}. 
	\end{rmk}

	\subsection{Summary}
	The error bounds in Theorem \ref{thm:up-bound} and Proposition \ref{prop:up-bound}  for the CIR process follow now from Theorem \ref{thm:one}, Theorem \ref{thm:two}, Proposition \ref{prop:one-proof-a}
	and  $|x-f_3(y)| \leq |x-y|$ for $x \geq 0$, $y\in \mathbb{R}$. 
	
	\medskip
	
	\section{Log-Heston model: convergence rates}
	\label{sec:heston}
	In this section, we will in particular show that the results from Theorem \ref{thm:one}, Theorem \ref{thm:two}  and Proposition \ref{prop:one-proof-a}  carry over to a discretization of the log-Heston model, where the log-price process is discretized with the Euler scheme \eqref{eulerscheme}, that is
	\begin{equation}
		\begin{aligned} \label{heston-euler-2}
			\hat{x}_{t}=\hat{x}_{\eta(t)}&+\left(\mu-\frac{1}{2}\hat{v}_{\eta(t)}\right)(t-\eta(t))+\sqrt{\hat{v}_{\eta(t)}}\left(U_t-U_{\eta(t)}\right)
		\end{aligned}
	\end{equation}
	with 
	\begin{displaymath}
		U_t=\rho W_t + \sqrt{1-\rho^2} B_t, \qquad  t \in [0,T].
	\end{displaymath}
	Recall that our first set of Euler schemes is given by 
	\begin{equation}\label{first-2}
		\begin{aligned}
			\bar{v}_t&=\bar{v}_{\eta(t)}+\int_{\eta(t)}^{t}\kappa(\theta - f_2(\bar{v}_{\eta(s)})) ds + \sigma \int_{\eta(t)}^{t}\sqrt{f_3(\bar{v}_{\eta(s)})}dW_s,\\
			\hat{v}_t &= f_3\left(\bar{v}_t\right),
		\end{aligned}
	\end{equation}where
	\begin{equation} \label{cir-choices1-2} 
		\begin{aligned}
		  f_2 \in \{ {\tt id,  abs, sym} \}, \qquad 
			f_3 \in \{ {\tt abs, sym} \},
		\end{aligned}
	\end{equation}
	while SE and AE can be expressed as
	\begin{equation}\label{helpsym-2}
		\begin{aligned}
			\hat{v}^{\star}_t=&\hat{v}_{\eta(t)}^{\star} +\int_{\eta(t)}^{t} \kappa\left(\theta-\hat{v}_{\eta(s)}^{\star}\right)ds+\sigma\int_{\eta(t)}^{t}\sqrt{\hat{v}_{\eta(s)}^{\star}}dW_s\\
			&-2c^{\star}\sigma\int_{\eta(t)}^{t}\mathbbm{1}_{\{z_s^{\star} \le 0\}} \sqrt{\hat{v}_{\eta(s)}^{\star}}dW_s-2c^{\star}\int_{\eta(t)}^{t}\mathbbm{1}_{\{z_s^{\star} \le 0\}} \kappa\left(\theta-\hat{v}_{\eta(s)}^{\star}\right)ds\\
			&+c^{\star}\left( L_t^0(z^{\star})-L_{\eta(t)}^0(z^{\star})\right)
		\end{aligned}
	\end{equation}
	where $c^{sym}=1$ and $c^{abs}=\frac{1}{2}$.

	\subsection{Error analysis - Case I}

	The key ingredient here and also for the second case is the observation that two continuous martingales $M=(M_t)_{t \in [0,T]}$
	and $\tilde{M}=(\tilde{M}_t)_{t \in [0,T]}$, whose quadratic variation coincides, have
	equivalent moments for their supremum norm.
	This directly follows from Proposition
	\ref{prop:martingalemoments}.

	\begin{theorem}\label{thm:three}
		Let $(\hat{x},\hat{v})$ be given by \eqref{heston-euler-2}, \eqref{first-2} and \eqref{cir-choices1-2}. Then, for all $\epsilon >0$ there exists a constant $C_{\epsilon}>0$  such that
		\begin{displaymath}
			\mathbb{E}\left[ 	\sup\limits_{t\in[0,T]} \left|X_t-\hat{x}_t\right|\right]  \leq C_{\epsilon} \left(\Delta t\right)^{\frac{\min\{1,\nu\}}{2}-\epsilon}.
		\end{displaymath}
	\end{theorem}
	\begin{proof}
		(i) Without loss of generality, we can assume $\mu=0$. We have to analyze
		\begin{displaymath}
			\mathbb{E}\left[ 	\sup\limits_{t\in[0,T]} \left|X_t-\hat{x}_t\right|\right] =\mathbb{E}\left[ 	\sup\limits_{t\in[0,T]} \left|\frac{1}{2}\int_{0}^{t}\left(\hat{v}_{\eta(s)}-V_s\right)ds+\int_{0}^{t}\left(\sqrt{V_s}-\sqrt{\hat{v}_{\eta(s)}}\right)d U_s\right|\right].
		\end{displaymath} Using Theorem \ref{thm:one}, Proposition \ref{prop:one-proof-a} and Lemma \ref{boundedscheme}, we obtain
		\begin{align}
				\mathbb{E}\left[ 	\sup\limits_{t\in[0,T]} \left|X_t-\hat{x}_t\right|\right]  
				&\le \frac{1}{2}\int_{0}^{T}\mathbb{E}\left[  \left|V_s-\hat{v}_{s}\right|\right] ds+\frac{1}{2}\int_{0}^{T}\mathbb{E}\left[ \left|\hat{v}_s-\hat{v}_{\eta(s)}\right|\right] ds+\mathbb{E}\left[ 	\sup\limits_{t\in[0,T]}\left|M_t\right|\right]  \nonumber \\
				&\le C_{\varepsilon} \left(\Delta t\right)^{\frac{\min\{1,\nu\}}{2}-\varepsilon} +\mathbb{E}\left[ 	\sup\limits_{t\in[0,T]}\left|M_t\right|\right] , \label{heston1}
			\end{align}
	
		where 
		\begin{displaymath}
			M_t = \int_{0}^{t}\left(\sqrt{V_s}-\sqrt{\hat{v}_{\eta(s)}}\right)dU_s.
		\end{displaymath}
		
		(ii) Let
		\begin{displaymath}
			\tilde{M}_t=
			\int_{0}^{t}\left(\sqrt{V_s}-\sqrt{\hat{v}_{\eta(s)}}\right)dW_s, \qquad t \in [0,T].
		\end{displaymath}
		Clearly, we have
		$$  \langle M \rangle_t = \langle \tilde{M} \rangle_t, \qquad  t \in [0,T], $$ and so Proposition 	
		\ref{prop:martingalemoments} yields
		\begin{equation} \label{mom-eq} 
		\mathbb{E}\left[ 	\sup\limits_{t\in[0,T]}\left|M_t\right|\right]
				\leq K_{1/2} \,  \mathbb{E}	\left[   \langle M \rangle_T^{\frac{1}{2}} \right] = K_{1/2} \,  \mathbb{E}	\left[   \langle \tilde{M} \rangle_T^{\frac{1}{2}} \right] 
			\leq \frac{K_{1/2}}{k_{1/2}} \,  	\mathbb{E}\left[ 	\sup\limits_{t\in[0,T]} |\tilde{M}_t|\right]. \end{equation}	
		Now, the Lyapunov inequality and an application of Doob's maximal inequality, i.e. Proposition \ref{prop:doob}, give
			\begin{equation} \label{mom-eq} 
			\mathbb{E}\left[ 	\sup\limits_{t\in[0,T]}\left|M_t\right|\right]
			\leq \frac{K_{1/2} (1+\beta)}{k_{1/2} \beta} \,  \left(	\mathbb{E}\left[  |\tilde{M}_T|^{1+\beta}\right] \right)^{1/(1+\beta)} \end{equation} for $\beta >0$.	
		
		Using   \eqref{first-2} and the SDE for the CIR process we have 
		\begin{equation*}
			\tilde{M}_T=
			\frac{1}{\sigma}\left(V_T-\overline{v}_T+\kappa\int_{0}^{T}\left(V_s-f_2(\overline{v}_{\eta(s)})\right)ds\right).
		\end{equation*} 
	Thus,	we obtain 
		\begin{equation}
			\begin{aligned} \label{log-heston_final}
				\mathbb{E}\left[ |\tilde{M}_T|\right] &\le\frac{1}{\sigma}\left(\mathbb{E}\left[ \left|V_T-\overline{v}_T\right|\right] +\kappa\int_{0}^{T}\mathbb{E}\left[ \left|V_s-\overline{v}_s\right|\right] ds+\kappa\int_{0}^{T}\mathbb{E}\left[ \left|\overline{v}_s-\overline{v}_{\eta(s)}\right|\right] ds\right)\\
				&\le C_{\varepsilon}\left(\Delta t\right)^{\frac{\min\{1,\nu\}}{2}-\varepsilon},
			\end{aligned}
		\end{equation} 
		where we used Theorem \ref{thm:one}, Proposition \ref{prop:one-proof-a}, Lemma \ref{boundedscheme} and the properties of $f_2$.
		Moreover, for all $p \geq 1$ there exists a constant $C_p >0$ such that
		\begin{displaymath}
			\mathbb{E}\left[ |\tilde{M}_T|^p\right]  \leq C_p
		\end{displaymath}
		due to  Lemma \ref{boundedscheme} and Lemma \ref{bounded}.
		Thus, a standard application of H\"older's inequality yields
		\begin{displaymath}
			\mathbb{E}\left[ |\tilde{M}_T|^{1+\beta}\right]  \leq  C_{\beta}  \,  \left(  \mathbb{E} \left[ |\tilde{M}_T| \right ] \right)^{\frac{1}{1+\beta}},
		\end{displaymath}
		which in turn together with \eqref{mom-eq} and  \eqref{log-heston_final} gives 
		\begin{equation} \label{heston-M-finale}
			\mathbb{E}\left[ 	\sup\limits_{t\in[0,T]}\left|M_t\right|\right] \leq C_{\beta,\varepsilon}   \left(\Delta t\right)^{ \left( \frac{ \min\{1,\nu\}}{2} -\varepsilon \right) \frac{1}{(1+\beta)^2}}.
		\end{equation} 
		
		(iii)  The assertion follows now from \eqref{heston1} and  \eqref{heston-M-finale}
		by choosing $\varepsilon$ and $\beta$ sufficiently small.
	\end{proof}

	\subsection{Error analysis - Case II}
	
	The second case can be treated analogously for $\nu>1$, except  at one point.
	Here the martingale $\tilde{M}$ is given by
	\begin{displaymath}
		\begin{aligned}
			\tilde{M}_t&= \int_{0}^{t}\left(\sqrt{V_s}-\sqrt{\hat{v}^{\star}_{\eta(s)}}\right)dW_s\\
			&=\frac{1}{\sigma}\left(V_t-\hat{v}_t^{\star}+\kappa\int_{0}^{t}\left(V_s-\hat{v}^{\star}_{\eta(s)}\right)ds\right)+2c^{\star}\int_{0}^{t}\mathbbm{1}_{\{z_s^{\star} \le 0\}} \sqrt{v_{\eta(s)}^{\star}}dW_s\\
			&\quad + \frac{2c^{\star}}{\sigma}\int_{0}^{t}\mathbbm{1}_{\{z_s^{\star} \le 0\}} \kappa\left(\theta-v^{\star}_{\eta(s)}\right)ds- \frac{c^{\star}}{\sigma} L^0_t(z^{\star}), \qquad t \in [0,T].
		\end{aligned}
	\end{displaymath}
	However, the additional  terms can be treated with Lemma  \ref{boundlocal} and Lemma \ref{Rem-SE-AE} and are (at least) of order $(\Delta t)^{\frac{1}{2}-\varepsilon}$.  Using Theorem \ref{thm:two} and Lemma \ref{boundedscheme2} instead of Theorem \ref{thm:one} and Lemma \ref{boundedscheme} and proceeding as in Case I we obtain 
	\begin{displaymath}
		\mathbb{E}\left[ |\tilde{M}_T| \right] \leq C_{\varepsilon} (\Delta)^{\frac{1}{2}-\varepsilon}.
	\end{displaymath}
	
	Therefore we also have the following result:
	\begin{theorem}\label{thm:four} 
		Let $\nu>1$ and $(\hat{x},\hat{v})$ be given by \eqref{heston-euler-2} and \eqref{helpsym-2}. Then, for all $\epsilon >0$ there exists a constant $C_{\epsilon}>0$  such that
		\begin{displaymath}
			\mathbb{E}\left[ 	\sup\limits_{t\in[0,T]} \left|X_t-\hat{x}_t\right|\right] \leq C_{\epsilon} \left(\Delta t\right)^{\frac{1}{2}-\epsilon}.
		\end{displaymath}
	\end{theorem}

	\subsection{Summary}
	The error bounds in Theorem \ref{thm:up-bound} and Proposition \ref{prop:up-bound}  for the log-Heston process follow now from Theorem \ref{thm:three} and Theorem \ref{thm:four}.
	
	\medskip

	\subsection{Proof of Proposition 1.3} 
	Recall that $\Delta t = \frac{T}{N}$. The Lipschitz continuity of $G$ implies that
	$$  \left| \mathbb{E} \left[   G \left(X^{pc,N}\right) \right]-   \mathbb{E} \left[  G \left(\hat{x}^{pc,N}\right) \right] \right|  \leq L_G \,
	\mathbb{E} \left[ \sup_{t \in [0,T]} |X_t -\hat{x}_t| \right],$$
	where we have set $X^{pc,N}_t= X_{\eta(t)}$ with $\eta(t)=\max\{k\in \{0,...,N\}:t_k\le t \} \Delta t$.
	Now Theorem  \ref{thm:three} and  Theorem \ref{thm:four}, respectively,  yield that 
	$$  \lim_{N \rightarrow \infty} \,\,   N^{1/2-\epsilon} \,  \big| \mathbb{E} \left[   G \left(X^{pc,N}\right) \right]-   \mathbb{E} \left[  G\left(\hat{x}^{pc,N} \right) \right] \big| =0$$ for $\nu>1$.
	Moreover, the Lipschitz continuity of $G$ also implies that
	$$  \big| \mathbb{E} \left[   G(X) \right]-   \mathbb{E} \left[  G \left(X^{pc,N}\right) \right] \big|  \leq L_G \, \mathbb{E} \left[ \sup_{t \in [0,T]} |X_t -X_{\eta(t)}| \right].$$
Finally, recall that due to Lemma $\ref{bounded-2}$ we have that
		$$ 	\sup_{0 \leq s< t \leq  T } \mathbb{E} \left[ \frac{|X_t-X_s|^p}{|t-s|^{p/2}} \right]<\infty $$
	for all $p \geq 1$. Thus, the classical Garsia-Rodemich-Rumsey Lemma, see e.g. Equation A.10 with $m=\gamma(\frac{1}{2}-\varepsilon)$ and $\gamma > \frac{1}{\varepsilon}$ in Appendix A.3  in \cite{Nualart}, implies that
	$$ 	 \mathbb{E} \left[ \sup_{0 \leq s< t \leq  T }  \frac{|X_t-X_s|}{ \, |t-s|^{\frac{1}{2}-\varepsilon}} \right]<\infty $$
for all $\varepsilon >0$.
	Therefore, we obtain
	$$  \lim_{N \rightarrow \infty} \,\,   N^{1/2-\epsilon} \, \big| \mathbb{E} \left[   G(X) \right]-   \mathbb{E} \left[  G \left(X^{pc,N}\right) \right] \big| = 0,$$
	which concludes the proof of  Proposition 1.3.
	
	\medskip
	
	\section{Simulation study}
	\label{sec:simulation}
	In this section, we will test numerically our results from Theorem \ref{thm:up-bound} and Proposition \ref{prop:up-bound} for some sets of exemplary parameters. We will perform the tests for all schemes from Table \ref{tab:Euler}. First, we will describe the design of the numerical experiments. We would like to estimate the rate of the decay of the errors
	\begin{displaymath}
		\begin{aligned}
			e_v(N) &= \mathbb{E}\left[ \left|V_T-\hat{v}_{t_N}^{(N)}\right|\right] , \qquad
			e_x(N) = \mathbb{E}\left[ \left|X_T-\hat{x}_{t_N}^{(N)}\right|\right] ,
		\end{aligned}
	\end{displaymath} 
	for the numerical scheme $\hat{v}^{(N)}, \hat{x}^{(N)}$ with step size $\Delta t =\frac{T}{N}$. Since we cannot compute these quantities exactly, we approximate their decay, see e.g. \cite{AA1}, by calculating
	\begin{displaymath}
		\begin{aligned}
			{\tt err}_v(N)&=\frac{1}{M}\sum_{i=1}^{M}\left|\left(\hat{v}_{t_N}^{(N)}-\hat{v}_{t_{2N}}^{(2N)}\right)^{(i)}\right|, \qquad
			{\tt err}_x(N)=\frac{1}{M}\sum_{i=1}^{M}\left|\left(\hat{x}_{t_N}^{(N)}-\hat{x}_{t_{2N}}^{(2N)}\right)^{(i)}\right|,
		\end{aligned}
	\end{displaymath} 
	where $M$ is the number of Monte Carlo repetitions and $(\hat{v}_{t_N}^{(N)}-\hat{v}_{t_{2N}}^{(2N)})^{(i)}, i=1,...,M$, are iid copies of $\hat{v}_{t_N}^{(N)}-\hat{v}_{t_{2N}}^{(2N)}$.  The same holds for $(\hat{x}_{t_N}^{(N)}-\hat{x}_{t_{2N}}^{(2N)})^{(i)}, i=1,...,M$. In our simulations, we chose $M=10^5$ and $N\in\{2^1,...,2^{15}\}$. To cover a wide range of different Feller indices, we will perform numerical simulations with five different parameter sets. We always choose $T=1$ and $s_0=100$. The other parameters can be found in Table \ref{tab:parameters}.
	\begin{table}[tbhp]
		{\footnotesize
			\caption{Parameters}\label{tab:parameters}
			\begin{center}
				\begin{tabular}{|c|c|c|c|c|c|c|c|} \hline
					Model & $v_0$ & $\kappa$ & $\theta$ & $\sigma$ & $\rho$ & $\mu$ & $\nu$\\ \hline
					1 & 0.04 & 5 & 0.04 & 0.61 & -0.7 & 0.0319 & 1.075\\
					2 & 0.0457 & 5.07 & 0.0457 & 0.48 & -0.767 & 0 & 2.0113 \\ 
					3 & 0.04 & 2.6 & 0.04 & 0.2 & -0.6 & 0 & 5.2\\
					4 & 0.010201 & 6.21 & 0.019 & 0.61 & -0.7 & 0.0319 & 0.63\\
					5 & 0.09 & 2 & 0.09 & 1 & -0.3 & 0.05 & 0.36\\
					\hline
				\end{tabular}
			\end{center}
		}
	\end{table}
	The  estimates ${\tt err}_v(N)$ and ${\tt err}_x(N)$ for the five Euler schemes are plotted in Figures \ref{fig:fig1}-\ref{fig:fig5} against the corresponding (inverse) step sizes of $2N$. For each model, we show first the convergence behavior of the error  for the CIR process and then for the Heston model. Additionally, we plotted a reference line with a suitable slope together with the error estimates. We also estimated the rate of convergence by the slope of a least squares fit, see Tables \ref{tab:model1} -- \ref{tab:model5}. Here, we only take errors with   (inverse) step sizes $N\in\{2^6,...,2^{15}\}$ into account to get a stable result. For all models, our simulation study shows that the numerical convergence rates do not change significantly, if we extend the simulation from the CIR process to the Heston model. This indicates that the parameters of the CIR process and especially the Feller index solely determine the convergence behavior. 
	\begin{figure}[tbhp]
		\centering
		\subfloat[CIR, $\nu=1.075$]{\label{fig:a}\includegraphics[scale=.21]{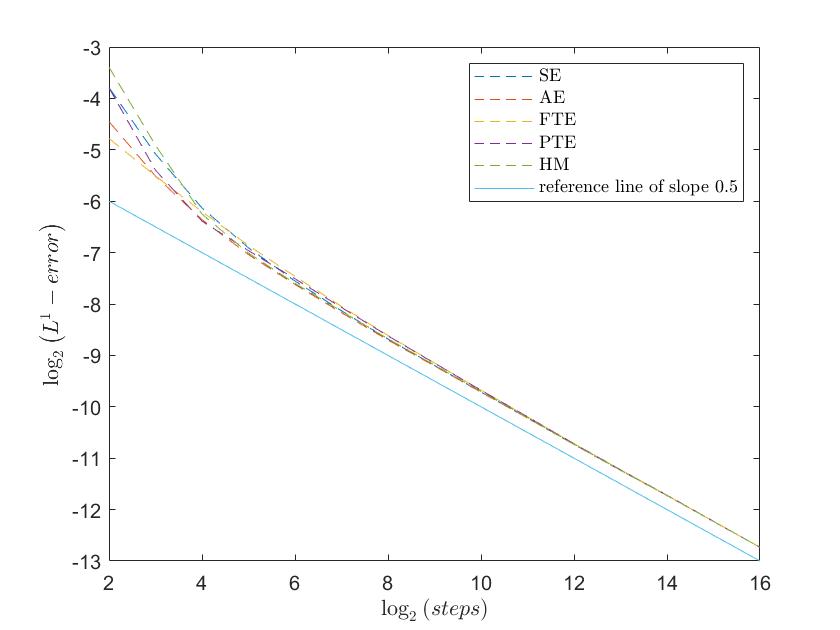}}
		\subfloat[Heston, $\nu=1.075$]{\label{fig:b}\includegraphics[scale=.21]{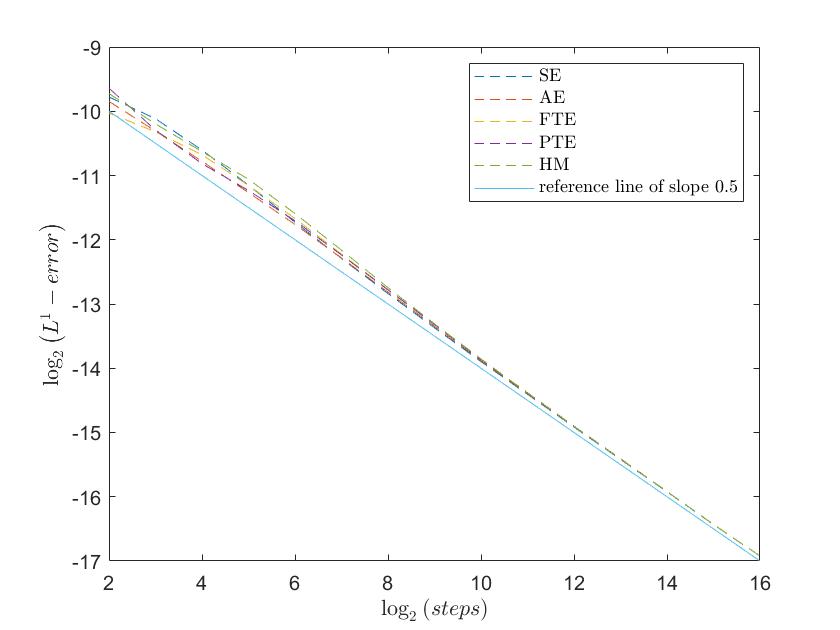}}
		\caption{Error estimates for Model 1}
		\label{fig:fig1}
	\end{figure}
	\begin{table}[tbhp]
		{\footnotesize
			\caption{Estimated convergence rates model 1}\label{tab:model1}
			\begin{center}
				\begin{tabular}{|c|c|c|} \hline
					Scheme & Rate CIR & Rate Heston \\ \hline
					SE & 0.5139 & 0.5175 \\
					AE & 0.5089 & 0.5171 \\ 
					FTE & 0.5236 & 0.5249 \\
					PTE & 0.5204 & 0.5226\\
					HM & 0.5117 & 0.5318\\
					\hline
				\end{tabular}
			\end{center}
		}
	\end{table}
	
	For the first model which has a Feller index around $1$, our main result provides a strong convergence rate of $0.5$. This can be numerically confirmed in Figure \ref{fig:fig1}.
	\begin{figure}[tbhp]
		\centering
		\subfloat[CIR, $\nu=2.0113$]{\label{fig:c}\includegraphics[scale=.21]{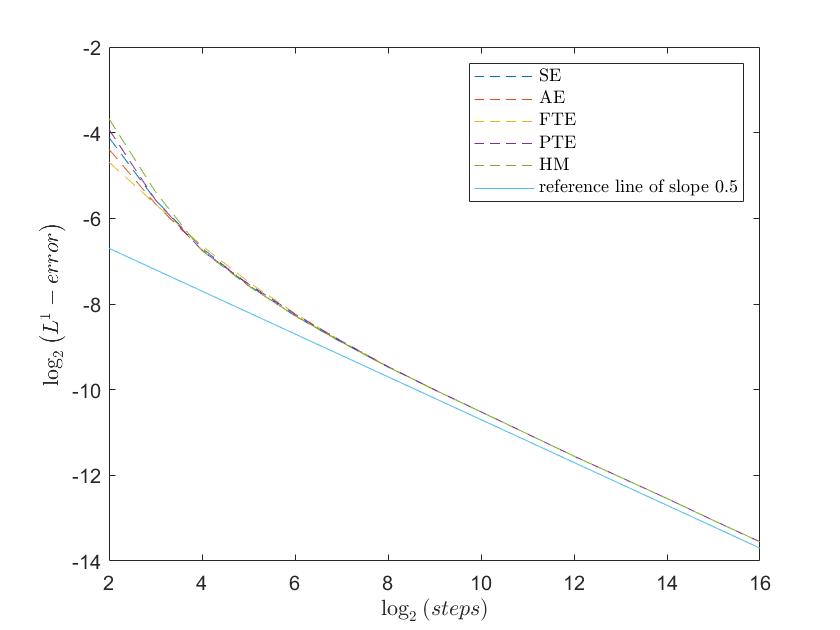}}
		\subfloat[Heston, $\nu=2.0113$]{\label{fig:d}\includegraphics[scale=.21]{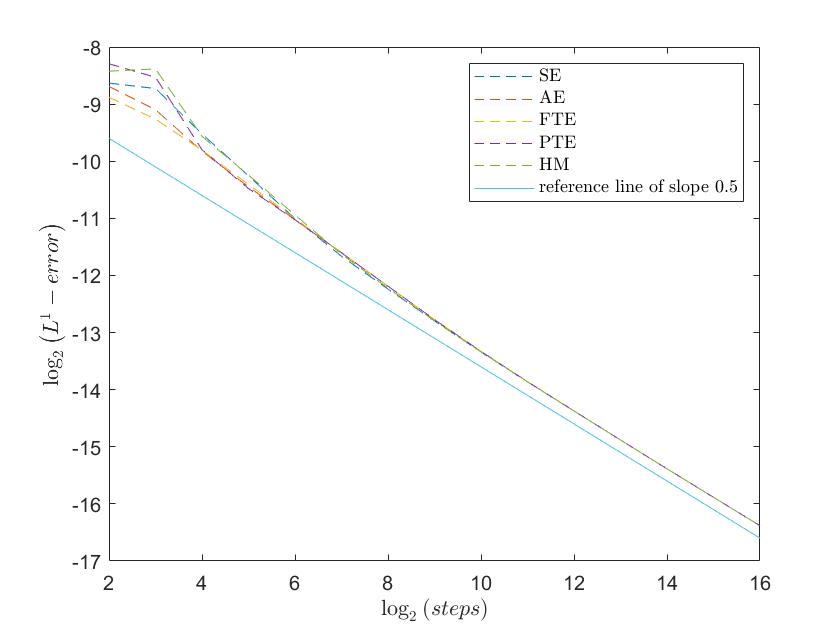}}
		\caption{Error estimates for Model 2}
		\label{fig:fig2}
	\end{figure}
	\begin{table}[tbhp]
		{\footnotesize
			\caption{Estimated convergence rates model 2}\label{tab:model2}
			\begin{center}
				\begin{tabular}{|c|c|c|} \hline
					Scheme & Rate CIR & Rate Heston \\ \hline
					SE & 0.5213 &  0.5310\\
					AE & 0.5216 &  0.5342\\ 
					FTE & 0.5251 &  0.5354\\
					PTE & 0.5243 & 0.5342\\
					HM & 0.5222 & 0.5357\\
					\hline
				\end{tabular}
			\end{center}
		}
	\end{table}
	\begin{figure}[tbhp]
		\centering
		\subfloat[CIR, $\nu=5.2$]{\label{fig:e}\includegraphics[scale=.21]{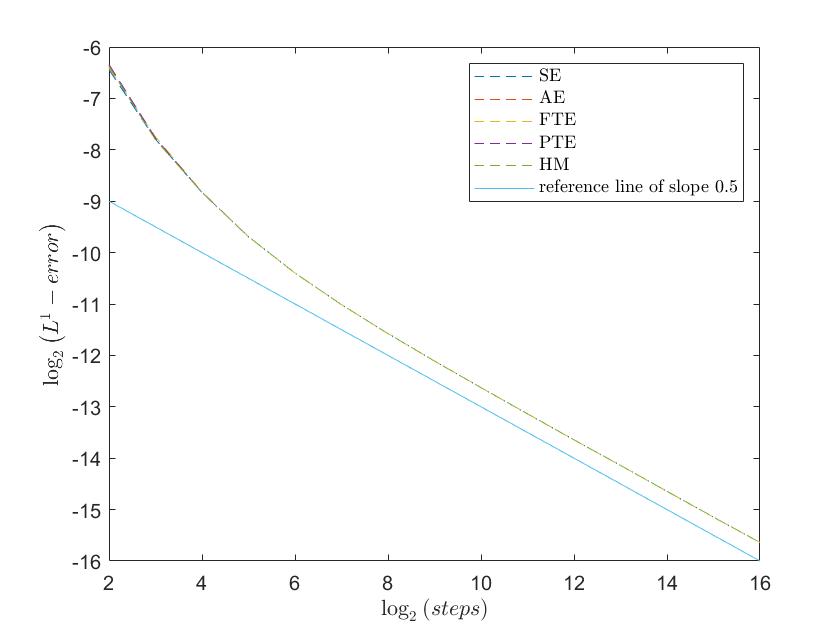}}
		\subfloat[Heston, $\nu=5.2$]{\label{fig:f}\includegraphics[scale=.21]{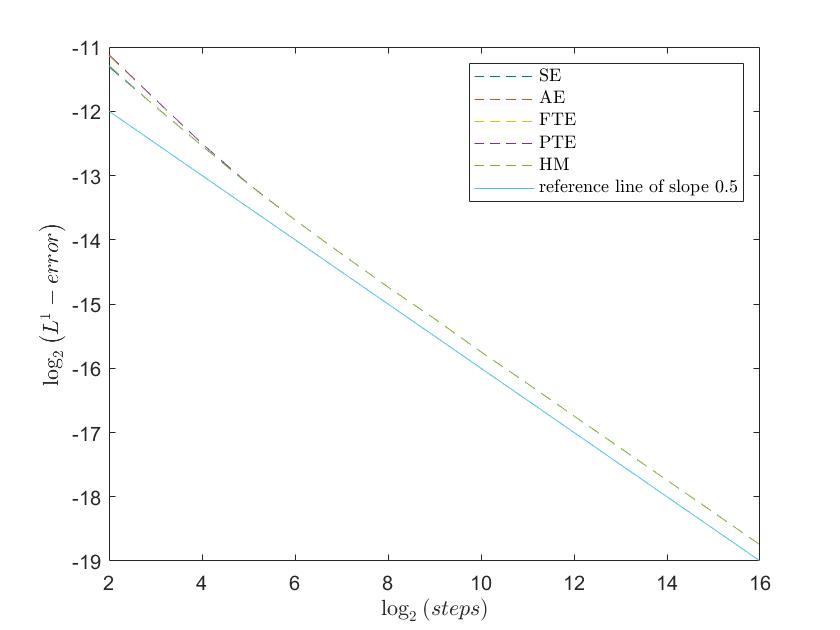}}
		\caption{Error estimates for Model 3}
		\label{fig:fig3}
	\end{figure}
	\begin{table}[tbhp]
		{\footnotesize
			\caption{Estimated convergence rates model 3}\label{tab:model3}
			\begin{center}
				\begin{tabular}{|c|c|c|} \hline
					Scheme & Rate CIR & Rate Heston \\ \hline
					SE & 0.5184 &  0.5038\\
					AE & 0.5184 &  0.5038\\ 
					FTE & 0.5184 &  0.5038\\
					PTE & 0.5184 & 0.5038\\
					HM & 0.5184 & 0.5038\\
					\hline
				\end{tabular}
			\end{center}
		}
	\end{table}
	
	For models 2 and 3, which have higher Feller indices, Figure \ref{fig:fig2} and Figure \ref{fig:fig3} confirm again the expected strong convergence rate of $0.5$. Note that the differences between the different Euler schemes vanish for small step sizes and for high Feller indices. The Euler schemes only differ, if the approximation of the CIR process becomes negative. For small step sizes and for high Feller indices this is unlikely to happen in a Monte Carlo simulation.
	\begin{figure}[tbhp]
		\centering
		\subfloat[CIR, $\nu=0.63$]{\label{fig:g}\includegraphics[scale=.21]{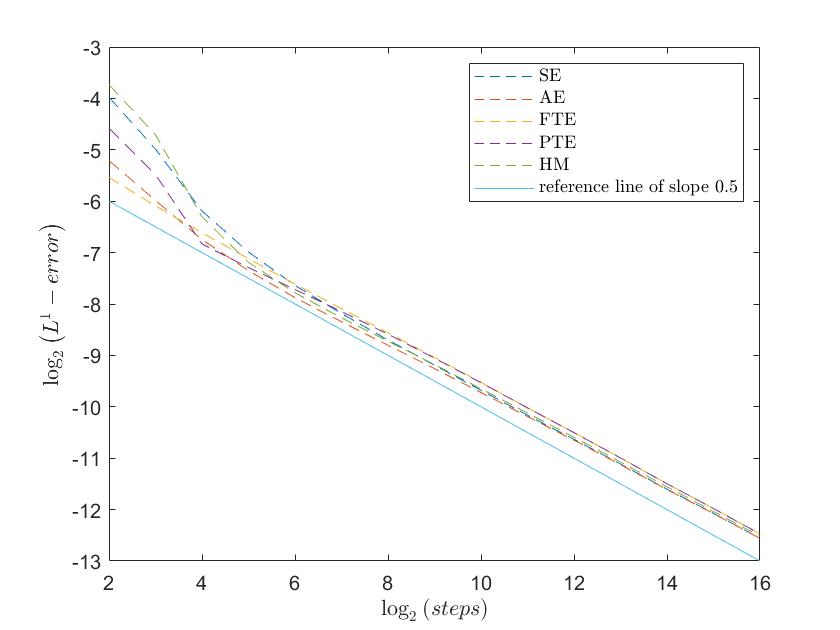}}
		\subfloat[Heston, $\nu=0.63$]{\label{fig:h}\includegraphics[scale=.21]{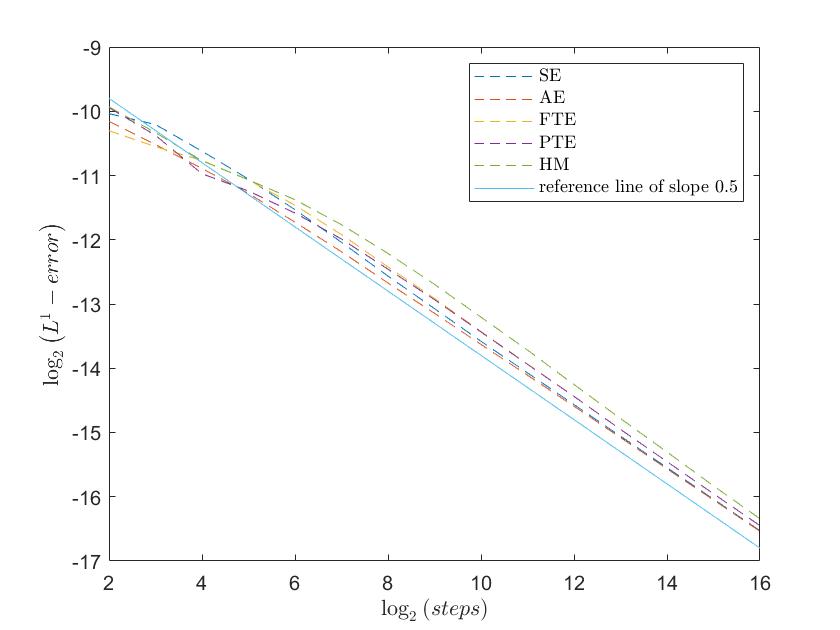}}
		\caption{Error estimates for Model 4}
		\label{fig:fig4}
	\end{figure}
	\begin{table}[tbhp]
		{\footnotesize
			\caption{Estimated convergence rates model 4}\label{tab:model4}
			\begin{center}
				\begin{tabular}{|c|c|c|} \hline
					Scheme & Rate CIR & Rate Heston \\ \hline
					SE & 0.4878 &  0.4988\\
					AE & 0.4675 &  0.4819\\ 
					FTE & 0.4868 &  0.5023\\
					PTE & 0.4790 & 0.4929\\
					HM & 0.4726 & 0.5053\\
					\hline
				\end{tabular}
			\end{center}
		}
	\end{table}
	
	Model 4 has a Feller index of $0.63$ and we have shown that we can expect a strong convergence rate of at least $0.315$ for FTE, PTE and HM. However, looking at Figure \ref{fig:fig4} we can see that the rate is still around $0.5$, even for SE and AE, for which we did not derive a convergence rate in this case.
	\begin{figure}[tbhp]
		\centering
		\subfloat[CIR, $\nu=0.36$]{\label{fig:i}\includegraphics[scale=.21]{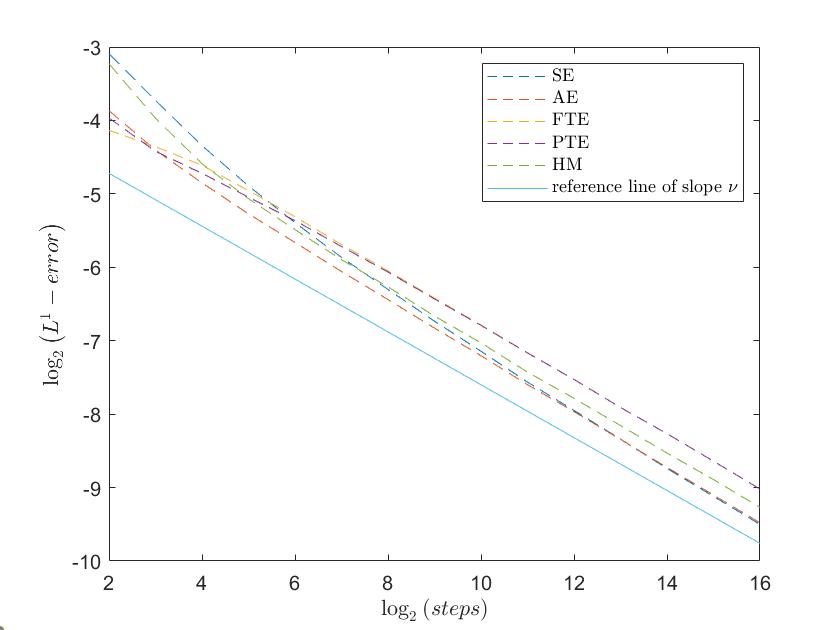}}
		\subfloat[Heston, $\nu=0.36$]{\label{fig:j}\includegraphics[scale=.21]{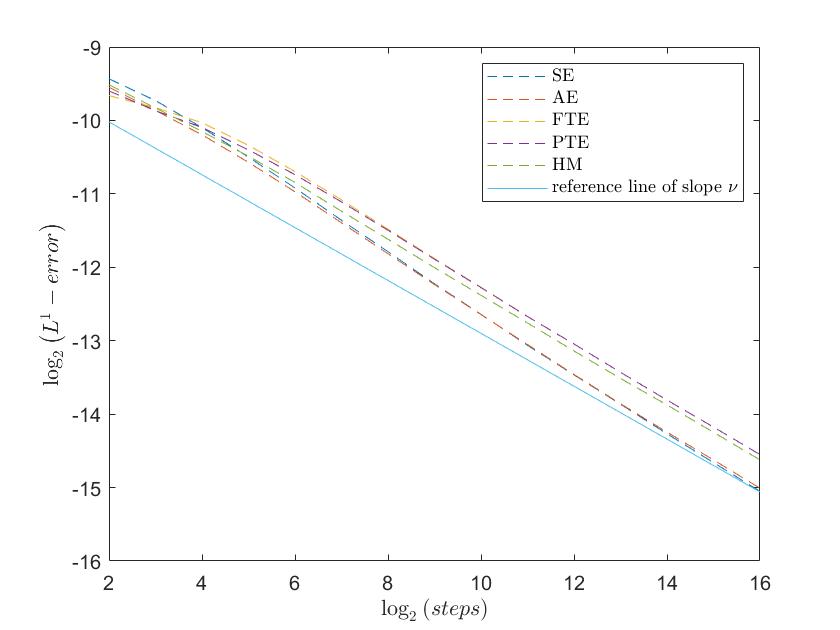}}
		\caption{Error estimates for Model 5}
		\label{fig:fig5}
	\end{figure}
	\begin{table}[tbhp]
		{\footnotesize
			\caption{Estimated convergence rates model 5}\label{tab:model5}
			\begin{center}
				\begin{tabular}{|c|c|c|} \hline
					Scheme & Rate CIR & Rate Heston \\ \hline
					SE & 0.4082 & 0.4128\\
					AE & 0.3811 & 0.4042\\ 
					FTE & 0.3701 & 0.3863 \\
					PTE & 0.3661 & 0.3828\\
					HM & 0.3765 & 0.3770\\
					\hline
				\end{tabular}
			\end{center}
		}
	\end{table}
	
	The last model has the lowest Feller index. Again, we can see an estimated convergence rate of the error that is better than expected. Here, we chose $\nu$ as the slope of the reference line.
	
	 The last two examples indicate that it might be possible to obtain a convergence rate of $\min\{\nu,\frac{1}{2}\}$ for all the considered Euler schemes.
	
	\bigskip
	\bigskip

	{\bf Acknowledgments.} \,\, 
	{\it The authors are grateful to the referees for their insightful comments and remarks, which helped to improve this article.}
	{\it  Annalena Mickel  has been supported by the DFG 
		Research Training Group 1953 "Statistical Modeling of Complex Systems".}

\end{document}